\documentclass[11pt,reqno]{amsart}
\usepackage{amsmath,amsthm,amssymb,amsfonts,enumerate,color,esint, graphicx, pgf,tikz,comment,hyperref}
\usepackage[utf8]{inputenc}
\usepackage{appendix}

\usepackage{xcolor}
\usetikzlibrary{calc}

\usepackage{fullpage} 

\usepackage{cite} 

\newcommand{\R}{\mathbb{R}}
\newcommand{\Rn}{\mathbb{R}^n}

\newcommand{\K}{\mathcal{K}}

\newcommand{\A}{\mathcal{A}}
\newcommand{\I}{\mathcal{I}}
\newcommand{\J}{\mathcal{J}}
\renewcommand{\H}{\mathcal{H}}
\renewcommand{\S}{\mathcal{S}}

\newcommand{\tr}{\operatorname{tr}}

\definecolor{darkgreen}{RGB}{0,100,0}

\newtheorem{thm}{Theorem}[section]
\newtheorem{prop}[thm]{Proposition}
\newtheorem{cor}[thm]{Corollary}
\newtheorem{lem}[thm]{Lemma}
\theoremstyle{definition}
\newtheorem{defn}[thm]{Definition}
\newtheorem{rem}[thm]{Remark}

\numberwithin{equation}{section}

\title{A geometric approach to nonlocal 2-Hessian equations}

\author[F. Charro]{Fernando Charro}
\address{Fernando Charro. Department of Mathematics, Wayne State University, 656 W. Kirby, Detroit, MI 48202, USA}
\email{fcharro@wayne.edu}
\author[M. Gonz\'alez]{Mar\'ia del Mar Gonz\'alez}
\address{Mar\'ia del Mar Gonz\'alez. Departamento de Matem\'aticas\\
 Universidad Aut\'onoma de Madrid\\
Campus de Cantoblanco, 28049 Madrid, Spain}
 \email{mariamar.gonzalezn@uam.es}
\author[M.~Soria-Carro]{Mar\'ia Soria-Carro}
\address{Mar\'ia Soria-Carro. Departamento de Matem\'aticas\\
 Universidad Aut\'onoma de Madrid\\
Campus de Cantoblanco, 28049 Madrid, Spain}
 \email{maria.soriacarro@uam.es}

\date{}

\begin{document}

\begin{abstract}
We study a nonlocal 2-Hessian equation, given by an infimum of linear deformations of the fractional Laplacian, that is,
$
\inf_{A\in \mathcal{A}_2}\Delta^s (u\circ A)(A^{-1}x)
$.
 We characterize the class of coefficient matrices $\mathcal{A}_2$, which determines the behavior of the operator, and provide a detailed geometric description of the possible degeneracies.
Our main theorem shows that the nonlocal 2-Hessian equation remains uniformly elliptic for strictly positive right-hand sides, which leads to regularity estimates.  
The results hold under weaker hypotheses than those previously considered in the literature, and for the full range $s\in(0,1)$. In particular, no convexity on the solutions is required. Our hypotheses can be interpreted as nonlocal counterparts of the local notions of  semiconcavity and 
2-convexity.
Moreover, all the results are stable as $s\to1$, recovering the local case.
The geometric methods developed here are new, even in the local setting, and may be relevant to a broader class of nonlocal fully nonlinear equations and curvature-type problems.

\end{abstract}

\maketitle

\section{Introduction}\label{sec:intro}

Many geometric problems, particularly those involving curvature of manifolds, are governed by nonlinear PDEs that are elliptic only within specific admissible classes of functions. 
 A classical example is the Monge-Amp\`ere equation $\det(D^2u)=f,$ which arises in the celebrated Minkowski problem of prescribed Gaussian curvature, and is elliptic precisely in the class of convex functions, see \cite{Figalli2017,Gutierrez}. 
The different classical notions of curvature for hypersurfaces are usually related in one way or another to symmetric polynomials in the principal curvatures, i.e.,
\[
\sigma_k(\kappa_1,\dots,\kappa_n)=\sum_{1\leq i_1<\cdots<i_k\leq n}\kappa_{i_1}\kappa_{i_2}\dots\kappa_{i_k},
\qquad k=1,2,\ldots,n.
\]
Of special interest are the mean curvature ($k=1$), scalar curvature ($k=2$), and Gaussian curvature ($k=n$). In particular, the intermediate Christoffel-Minkowski problem generalizes the problem of finding a convex hypersurface with prescribed Gauss curvature on its outer normal, to prescribed $\sigma_k$ curvature (see the survey \cite{Guan-Ma} and the references therein).

When a hypersurface is given by a graph $z=u(x)$, and upon removing the tangent plane, these principal curvatures can be seen as the eigenvalues of the Hessian matrix $D^2u$. This leads naturally to the so-called $k$-Hessian equations
\begin{equation}\label{k.hessian.local}
\sigma_k\big(\lambda(D^2u)\big)=f, \qquad k=1,2,\ldots,n,
\end{equation}
where $\lambda(D^2u)=\big(\lambda_1(D^2u),\ldots,\lambda_n(D^2u)\big)$ denotes the eigenvalues of $D^2u$, counted with multiplicity. The $k$-Hessian operators are (degenerate) elliptic in the cone 
\begin{equation}\label{k-cone}
\Gamma_k=\Big\{\lambda\in\R^n: \sigma_j(\lambda)> 0\ \text{for all}\ j=1,\dots,k \Big\},
\end{equation}
which introduces a natural geometric restriction for the solutions of \eqref{k.hessian.local}, i.e., that $D^2u\in~\Gamma_k$ (in the sense that $\lambda(D^2u)\in \Gamma_k$), also known as \emph{$k$-convexity} of the solutions. When $u$ is $k$-convex, we have $\sigma_k\big(\lambda(D^2u)\big)\geq 0$. Hence, in the study of $k$-Hessian equations, one always assumes that $f\geq 0$. 
If, in addition, $f\geq \eta_0>0$, then the equation \eqref{k.hessian.local} becomes uniformly elliptic, and classical solutions are obtained by proving an a priori $C^{2,\alpha}$ estimate in order to apply the method of continuity. Some of the most influential works are \cite{Ca-Ni-Sp,Ca-Ni-Sp2,Ivochkina,Trudinger,TW1,TW2,TW3}, although the literature is extensive; see also the survey \cite{Wang-survey} and the references therein.

Over the past years, intense research efforts have focused on establishing nonlocal and fractional counterparts of the notion of curvature, see for instance \cite{Abatangelo-Valdinoci,[Caffarelli et al. 2010],[Chang and Gonzalez 2011]}, although the theory is still under development.
A natural step is to consider nonlocal analogues of the $k$-Hessian operators in \eqref{k.hessian.local}.

Given the various applications of the Monge-Amp\`ere operator, $\det(D^2u)$, in geometry and analysis (e.g., Gaussian curvature, affine geometry, and optimal transport, see \cite{Figalli2017,Gutierrez}), the lack of a nonlocal counterpart was a significant gap in the field of nonlocal PDEs that \cite{Caffarelli-Charro,Caffarelli-Silvestre-16,Guillen-Schwab,Jhaveri-Stinga,Maldonado.Stinga,Stinga-Vaughan} contributed to bridge. 
The motivation of these nonlocal Monge-Amp\`ere operators comes from the following characterization of the determinant, which can be regarded as a matrix version of the arithmetic-geometric mean inequality: For every $n\times n$ symmetric matrix $M\geq0$, we have
\begin{equation*}
n\left(\det{M}\right)^{1/n}=\inf_{\det A=1} \tr(A^tMA),
\end{equation*}
see for instance \cite{BCMR}.
Considering the left polar decomposition of $A$, one can assume that the matrices $A$ are symmetric and positive definite. Based on this observation, we will always assume the matrices $A$ to be symmetric.
In particular, for any convex $C^2$-function~$u$, 
\begin{equation}\label{inf.M.A.local}
n\, \sigma_n^{1/n}\big(\lambda(D^2u)\big) = n\,(\det D^2u )^{1/n}=\inf_{\det A=1}\mathcal L_{A}u,
\end{equation}
where $$\mathcal L_Au=\tr(AD^2u\,A)=\Delta (u\circ A)(A^{-1}x).$$
Here, with a slight abuse of notation, we write $A$ for the linear mapping $x\mapsto Ax$.

Observe that $\sigma_n^{1/n}\big(\lambda(D^2u)\big)$ is a concave operator in the matrix argument, and \eqref{inf.M.A.local} can be seen as a linear envelope.
Then, \cite{Caffarelli-Charro} naturally defines the fractional Monge-Amp\`ere operator of order $s\in(0,1)$ as a linear envelope of the nonlocal counterparts of the $\mathcal{L}_A$, i.e.,
\begin{equation}\label{fractional.MA}
\mathcal{D}_s u(x)=\inf_{\det A=1}\mathcal L_{A}^s u(x)
\end{equation}
for 
\begin{equation} \label{eq:fraclinop}
\mathcal L_A^s u(x) =\Delta^s (u\circ A)(A^{-1}x)
=C_{n,s} \text{P.V.}\int_{\mathbb{R}^n}\frac{u(x+Ay)-u(x)}{|y|^{n+2s}}\,dy.
\end{equation}
Here, $\Delta^s$ is the fractional Laplacian (with the sign convention as in \eqref{defi-frac-lap}), and 
 \begin{equation} \label{eq:ctfraclapl}
 C_{n,s}:= \frac{4^{s} s \Gamma \big( \frac{n+2s}{2}\big)}{\pi^\frac{n}{2}\Gamma(1-s)}
 \end{equation}
is a normalization constant. 

Notice that the operator $\mathcal{D}_s$ is a degenerate Bellman operator since the condition $\det A=1$ allows the matrix $A$ to have arbitrarily small eigenvalues (this is easily seen in dimension 2, where the two eigenvalues of the matrix $A$ must be reciprocal); for a general reference on Bellman equations, see \cite{Krylov}.
In \cite{Caffarelli-Charro}, the authors identify conditions on $u$ under which the operator $\mathcal{D}_s u(x)$ remains strictly elliptic, which allows them to apply regularity results for uniformly elliptic operators such as the Evans-Krylov estimates (see \cite{Caffarelli.Silvestre,Caffarelli.Silvestre2}) and conclude that solutions are classical. Moreover, upon an appropriate normalization, $\mathcal{D}_s u(x)$ converges to the classical Monge-Amp\`ere operator as the order of the nonlocal equation approaches the local case, i.e., ~as ~$s\to1$.

\medskip
Our goal in this article is to develop a similar program for the $2$-Hessian equation from a geometric perspective. That is, we aim to define a nonlocal $2$-Hessian operator and study its degeneracy and natural solution class, under which we finally establish regularity results along the lines of \cite{Caffarelli-Charro}.
We seek to establish a foundation that can be used to define general nonlocal $k$-Hessian equations in the future.

This is a nontrivial task, and several approaches have been made in the literature. We cite the PhD thesis \cite{Wu:thesis} and the unpublished paper \cite{Wu:paper} for a possible definition of the nonlocal 2-Hessian operator. Finally, there is also the work \cite{Caffarelli-SC}, where they study a different family of nonlocal operators interpolating the fractional Laplacian and Monge-Amp\`ere operator. Nevertheless, the operators in this family do not share the same geometric structure expected of $k$-Hessian operators.\medskip

In the sequel, we write the local 2-Hessian operator in the form $\sigma_2^{1/2}\big(\lambda(D^2u)\big)$, or, simply, $\sigma_2^{1/2}(D^2u)$.
In this way, the operator is concave and homogeneous of degree 1, and one can hope to describe it as a linear envelope similarly to \eqref{inf.M.A.local}, by identifying an appropriate class of matrices $\mathcal{A}_2$ to replace $\{\det{A}=1\}$. 
Then, one can extend the definition to the nonlocal case as in \eqref{fractional.MA}.

In \cite[Lemma 1.4]{BCMR}, the authors characterize precisely this set $\mathcal A_2$, which allows them to~write
\begin{equation}\label{2-Hessian}
 2\, \sigma_2^{1/2}(D^2 u) = \inf_{A\in \A_2} \mathcal L_A u = \inf_{A\in \A_2} \tr (A D^2 u A)
\end{equation}
whenever $u$ is such that $D^2u\in\overline\Gamma_2$, using the notation in \eqref{k-cone}.
In particular,
\begin{equation*}
 \A_2 := \big\{ A \in \S^n_+: \lambda_i(A)^2 = \sigma_{1,i}(\zeta) \ \text{ with } \ \zeta\in \Gamma_2 \ \text{ and } \ \sigma_2(\zeta)=1 \big\}, 
\end{equation*}
where $\mathcal S_+^n$ is the set of symmetric positive definite matrices, $\zeta=(\zeta_1,\hdots, \zeta_n) \in \Rn$,
$$
\sigma_1(\zeta) = \sum_{i=1}^n \zeta_i \quad \text{ and } \quad
\sigma_2(\zeta)= \sum_{1 \leq i<j \leq n} \zeta_i \zeta_j,
$$
 and $\sigma_{1,i}(\zeta)=\sigma_1(\zeta_1, \hdots,\zeta_{i-1},0,\zeta_{i+1},\hdots, \zeta_n)$.
For the representation of general Hessian operators as infima over coefficient sets, see also \cite{Krylov-TAMS}.

Observe that $2$-convex functions are precisely those for which the infimum in \eqref{2-Hessian} is finite (and, in turn, nonnegative). In the literature, this is sometimes called \emph{admissibility} \cite{BCMR}.

\subsection{Main result}

Let $n\geq 3$. For $s\in (0,1)$, we consider a nonlocal 2-Hessian operator given by
\begin{equation}\label{frac-2-hessian}
 \K^s u(x) := \inf_{A\in \A_2} \mathcal L_A^s u(x) = C_{n,s} \inf_{A\in \A_2}  \text{P.V.}\int_{\Rn} \frac{u(x+Ay)-u(x)}{|y|^{n+2s}}\, dy, 
\end{equation}
where $C_{n,s}$ is the normalizing constant defined in \eqref{eq:ctfraclapl}.
To control the tails of the operator, we introduce the usual weight $\omega_s(y):=\big(1+|y|^{n+2s}\big)^{-1}$, and say that $u\in L^1_{\omega_s}(\Rn)$ if 
$$
\| u\|_{L^1_{\omega_s}(\Rn)} := \int_{\Rn} \frac{|u(y)|}{1+|y|^{n+2s}}\, dy<\infty.
$$
For $u\in C(\Rn)\cap L^1_{\omega_s}(\Rn)$, sub- and supersolutions of $\K^su=f$ are
well defined in the viscosity sense, and classically at any point where $u$ can be touched
 by an appropiate test function (see Definition~\ref{def:viscosity} and 
Lemma~\ref{lem:classicalsense}). 
When $s\to 1$, the operator $\K^s$ converges to the classical 2-Hessian operator \eqref{2-Hessian}. More precisely:

\begin{rem}\label{convergence.s.to.1}
 Let $s\in (0,1)$. Assume $u\in C^2(\Rn)\cap L_{\omega_s}^1(\Rn)$ is 2-convex, i.e., $D^2u\in\overline\Gamma_2$.
Then, 
 $$
 \lim_{s\to 1} \K^s u(x) = 2\, \sigma_2^{1/2}(D^2 u(x)),
 $$
 in the viscosity sense (see Definition \ref{def:viscosity}). 
 The proof is similar to the one given in \cite[Lemma~A.1]{Caffarelli-Charro}, thus we omit it here.

\end{rem}

Our main result states that the nonlocal 2-Hessian equation becomes uniformly elliptic for strictly positive right-hand sides. More precisely, we will show that the infimum cannot be attained at matrices that are too degenerate, i.e., matrices that have very small or very large eigenvalues. 

To make this precise, let us first introduce some notation.
Given a real number $\Lambda\geq 1$, we define the subset of \textit{nondegenerate (or uniformly elliptic)} positive symmetric matrices with respect to~$\Lambda$,~as
\begin{equation}\label{eq:nondeg}
\mathcal{N}(\Lambda) :=\big \{ A \in \S^n_+ : \Lambda^{-1} I \leq A \leq \Lambda I \big\}.
\end{equation}
Naturally, we can write 
$\A_2\cap\mathcal{N}(\Lambda)$ for the subset of matrices in $\A_2$ that are positive and nondegenerate with respect to $\Lambda$.

Define also, for $x, y\in \R^n$, the centered second difference 
$$
\delta u(x,y) := u(x+y)+u(x-y)-2u(x), 
$$
and the one-dimensional fractional Laplacian in the $e$-direction as
\begin{equation}\label{definition-h-s}
 \Theta^s u(x,e):= 
 (1-s)\int_{-\infty}^{\infty} \frac{\delta u(x,r e)}{|r|^{1+2s}}\, dr, \quad \text{for}\ e\in \mathbb{S}^{n-1}.
\end{equation}
The latter has appeared in several other works \cite{DelPezzo-Quaas-Rossi,Birindeli-Galise-Topp,Rossi-RuizCases}, as it allows for more general eigenvalue constructions.
 Here, we primarily consider its homogeneous extension of degree $2s$, given by
\begin{equation}\label{def.I.s.intro}
\I^s u(x,y):= (1-s)\int_{-\infty}^{\infty} \frac{\delta u(x,ry)}{|r|^{1+2s}}\,dr,
\quad\textrm{for}\ y\in \mathbb{R}^{n}.
\end{equation}
For $u\in C^2$, \eqref{definition-h-s} and \eqref{def.I.s.intro} are the nonlocal counterparts of the pure second derivative along a given direction and the Hessian quadratic form at a point $x\in\mathbb R^n$, i.e.,
$$
\Theta u(x,e):=u_{ee}(x),\quad \textrm{for}\ e\in\mathbb{S}^{n-1},
$$
and
\begin{equation}\label{def.I.local}
\I u(x,y):=\langle D^2u(x)y,y\rangle
=\left\{
\begin{aligned}
&|y|^2\,\Theta u\Big(x,\tfrac{y}{|y|}\Big),& & 
\textrm{for}\ y\in\Rn\setminus\{0\},\\
&0, &&\textrm{for}\ y=0,
\end{aligned}
\right.
\end{equation}
respectively. 
 We observe that for $u\in C^2\cap L^1_{\omega_s}$, the nonlocal objects converge to the corresponding local ones, i.e., 
\begin{equation}\label{convergence-Is}
\Theta^su(x,e)\to \Theta u(x,e) \quad  \text{ and } \quad  \I^s u(x,y)\to \I u(x,y) \quad\textrm{ as $s\to1$}.
\end{equation}
Motivated by \eqref{convergence-Is}, we refer to $\I^s u(x,\cdot)$ as the \emph{nonlocal Hessian form} of $u$ at $x$.

\medskip

Our main result is the following.

\begin{thm}[Uniform Ellipticity] \label{thm:main}
Consider a real number $s\in (0,1)$ and a domain $\Omega\subset\Rn$.
Let $u\in C(\Rn)\cap L_{\omega_s}^1(\Rn)$ be such that for some $\eta_0>0$ it holds
\begin{equation}\label{nonlocal.2.convexity}
 \K^s u(x) \geq \eta_0, \quad \text{ for } x\in \Omega,
\end{equation}
 in the viscosity sense. 
Moreover, assume there exists $K>0$ such that for every $x\in\Omega$ at which $u$ can be touched from above by a $C^2$-function it holds
\begin{equation}\tag{NSC}\label{NSC}
\mathcal Q^s u(x;y,z):=\I^s u(x,y+z)+ \I^s u(x,y-z) - 2\,\I^s u(x,y)
\leq 2K |z|^{2s}
\end{equation}
for a.e. $y,z\in\mathbb{R}^n$.
 Then, there is $\Lambda\geq 1$, depending only on $n$, $s$, $K$, and $\eta_0$, such that
\begin{equation}\label{uniform.Ks.main.thm}
 \K^s u(x) = C_{n,s} 
 \min_{A\in \A_2\cap\mathcal{N}(\Lambda)} {\rm P.V.}
 \int_{\Rn} \frac{u(x+Ay)-u(x)}{|y|^{n+2s}}\, dy,
 \end{equation}
 i.e., $\K^s$ is uniformly elliptic in $\Omega$, and the infimum is attained. Moreover, 
 $\Lambda=O(1)$ as $s\to 1$.
\end{thm}

\begin{rem}\label{remark.I.s.in.L1}
Observe that $u\in L^1_{\omega_s}(\mathbb{R}^n)$ and \eqref{nonlocal.2.convexity} imply $\I^s u(x,\cdot)\in L^1(\mathbb{S}^{n-1})$. Hence, by homogeneity, the  $\I^s u(x,y\pm z)$ in hypothesis \eqref{NSC} are finite for a.e. $y,z\in\Rn$. 
\end{rem}

The linear operators $\mathcal{L}_A^s$, defined in \eqref{eq:fraclinop}, can be written as
\begin{equation}\label{alternative.writing.L.A1}
\mathcal{L}_A^s u(x) = {\rm P.V.} \int_{\Rn} ( u(x+y)-u(x) ) \, k_A(y)\, dy,
\end{equation}
where the kernel $k_A$ is given by
\begin{equation}\label{alternative.writing.L.A2}
k_A(y) = \frac{C_{n,s}}{\det(A)|A^{-1}y|^{n+2s}}.
\end{equation}
In particular, it holds that
\begin{equation}\label{alternative.writing.L.A3}
\frac{\lambda_{\min}(A)^{n+2s}}{\lambda_{\max}(A)^n} \cdot \frac{ C_{n,s} }{|y|^{n+2s}} \leq k_A(y) \leq \frac{\lambda_{\max}(A)^{n+2s}}{\lambda_{\min}(A)^n} \cdot \frac{ C_{n,s} }{|y|^{n+2s}}.
\end{equation}
If $A\in \A_2\cap\mathcal{N}(\Lambda)$, then $k_A$ is comparable to the kernel of the fractional Laplacian, with constants depending only on the universal parameter $\Lambda$. 
 Hence, Theorem~\ref{thm:main} implies that $\K^s$ belongs to a class of nonlocal Bellman operators, for which
 the theory of nonlocal uniformly elliptic operators is available. 
As a consequence of Theorem \ref{thm:main}, the nonlocal analogues of the Krylov-Safonov theorem, and the Evans-Krylov theorem, we obtain the following regularity~result (see, for instance, the seminal works \cite{Caffarelli.Silvestre,Caffarelli.Silvestre2}, the monograph~\cite{FernandezReal-RosOton}, and the references therein).

\begin{cor}[Interior regularity]
 Given $s\in (0,1),$ and $K,\eta_0>0$, let $\Lambda\geq 1$ be as in Theorem~\ref{thm:main}.
 Then, there exists $\alpha_0>0$, depending only on $n$, $s$, and $\Lambda$, such that: 
 For all $\alpha \in (0, \alpha_0)$ and $f\in C^\theta(\Rn)$, with $\theta=\max\{1-2s,0\}+\alpha$, satisfying 
 $$f\geq \eta_0>0\quad \text{in } B_1(0),$$ 
 if $u\in C(\Rn)\cap L_{\omega_s}^1(\Rn)$
 satisfies \eqref{NSC}, and 
 $$
 \K^s u = f \quad \text{ in } B_1(0),
 $$
 in the viscosity sense, then there is $C>0$, depending only on $n$, $s$, $\alpha,$ and $\Lambda$, such that
 $$
 \|u\|_{C^{\beta+\alpha}(B_{1/2}(0))} \leq C \big( \|u\|_{L^1_{\omega_s}(\Rn)} + \|f\|_{C^{\theta}(\Rn)}\big),
 $$
 where $\beta=\max\{1,2s\}$. 
\end{cor}

\subsection{Discussion of the hypotheses}\label{subsection.discussion.hyp}

A few comments regarding Theorem \ref{thm:main} are in order. First of all, the hypotheses of the theorem force a compatibility condition among the data of the form
\begin{equation}\label{compatibility-intro}
 \eta_0\leq\K^s u(x)\leq c(n,s) K,
 \end{equation}
 which is stable as~$s\to1$ (see Lemma \ref{lemma:compatibility} for the exact constants and proof).

The fundamental assumption here is \eqref{NSC}. It can be regarded as a nonlocal semiconcavity condition (the exponent $2s$ on the right-hand side is dictated by the $2s$-homogeneity of $\I^su(x,\cdot)$). This is because for a function $u\in C^2\cap L^1_{\omega_s}$, the inequality \eqref{NSC} converges as $s\to 1$ to
\begin{equation}\label{intro.eq.local.1}
\I u(x,y+z)+ \I u(x,y-z) - 2\,\I u(x,y) \leq 2K |z|^{2},\qquad\textrm{for all $y,z\in\mathbb{R}^n$},
\end{equation}
see \eqref{convergence-Is}. After simplification, this amounts to the classical (local) semiconcavity condition 
 \begin{equation}\label{intro.eq.local.2}
 \I u(x,z)=\langle D^2u(x)z,z\rangle\leq K|z|^2\quad\textrm{for all $z\in\mathbb{R}^n$,}
 \end{equation}
 which can be formulated also as $\delta u (x,y)\leq K|y|^2$, for all $x$ and all sufficiently small $y\in \Rn$.

Semiconcavity, along with Lipschitz continuity and \mbox{$s>1/2$} (to control the growth at infinity), has been typically assumed in previous uniform ellipticity results in the literature, see \cite{Caffarelli-Charro,Gan-Jiao,Wu:paper}.
In contrast, \eqref{NSC} is a weaker condition, which, as we will show below, allows us to consider the case $s\leq 1/2$, never considered before.  

One could seek to extend the analogy with the local case, noting that conditions \eqref{intro.eq.local.1} and \eqref{intro.eq.local.2} are equivalent for $s=1$, and ask whether   condition \eqref{NSC} could be weakened to $\I^s u(x,z)\le K|z|^{2s}$, or, equivalently, 
 \begin{equation}\label{weaker-theta}
 \Theta^s u(x,e)\le K.
 \end{equation}  However, we will show in Section~\ref{sec:counterexample} that the conclusion of Theorem~\ref{thm:main} may fail under this weaker assumption. The reason is that $\mathcal Q^s u(x;\cdot,\cdot)$ couples the values of $\I^su(x,\cdot)$ across different directions, whereas the bound \eqref{weaker-theta} controls them only individually. The transversal information provided in \eqref{NSC} is essential in the proof, see~\eqref{lem:comp.eq1}. Further discussion on nonlocal semiconcavity will be continued in Section~\ref{sec:NSC.discussion}.

Nevertheless, \eqref{NSC} may be difficult to verify directly for a given function, since it is formulated for the integral quantities $\I^s u$ rather than $u$ itself. We will show in Lemma \ref{lemma:implication-semiconcavity} that condition \eqref{NSC} holds whenever $u$ satisfies the following semiconcavity with a modulus condition (see \cite{CS}),
 \begin{equation} \label{eq:onesidedreg} 
 \delta u(x,y)=u(x+y)+u(x-y)-2u(x) \leq m(|y|), \quad \text{ for all } x,y \in \Rn,
\end{equation}
for some modulus of continuity $m: [0,\infty)\to[0,\infty)$ with
\begin{equation} \label{eq:hyp.on.m}
2(1-s) \int_0^\infty \frac{m(r)}{r^{1+2s}}\,dr\leq K.
\end{equation}
In particular, upon adjusting the constant, \eqref{NSC}  holds for globally semiconcave and Lipschitz functions when $s>1/2$, with 
$m(r)=\min\{2Lr,\,Kr^2\}$ (where $L$ is the Lipschitz constant, which controls integrability at infinity), which is the setting of 
\cite{Caffarelli-Charro,Gan-Jiao,Wu:paper}; for $s\leq 1/2$, one may take instead $m(r)=\min \{L r^{2s-\kappa}, Kr^2\}$, for some $\kappa\in(0,2s)$.
Note that classical semiconcavity alone corresponds to $m(r)=Kr^2$, which fails to satisfy the integrability condition \eqref{eq:hyp.on.m} at infinity.

\medskip

On another note, hypothesis \eqref{nonlocal.2.convexity} can be regarded as a nonlocal counterpart of the classical (local) strict 2-convexity, and suggests a way to define nonlocal $k$-convexity.
A fractional notion of convexity (the case $k=n$) has been proposed in \cite{DelPezzo-Quaas-Rossi}, but
for intermediate values of $k$, it is not clear what the correct nonlocal analogue of $k$-convexity should be. 
Whenever $D^2u(x)\in\overline\Gamma_2$, we have the representation formula
\begin{equation}\label{rep.formula.local}
2\,\sigma_2^{1/2}\big(D^2u(x)\big)
=\frac{n}{\omega_{n-1}}\inf_{A\in\A_2}\int_{\mathbb S^{n-1}}\I u(x,Ae)\,d\H^{n-1}_e,
\end{equation}
see \eqref{2-Hessian} and \eqref{representation.local}.
Therefore, the function $u$ is 2-convex if and only if 
\[
\inf_{A\in\A_2}\int_{\mathbb S^{n-1}}\I u(x,Ae)\,d\H^{n-1}_e \geq0.
\]
This motivates a notion of nonlocal 2-convexity through the operator $\K^s$. Namely, a function $u$ is nonlocal 2-convex whenever it satisfies
\[
\K^s u(x) = 
\frac{C_{n,s}}{4(1-s)}\,
\inf_{A\in \A_2}\int_{\mathbb S^{n-1}}\I^su\big(x,\,A e \big)\,d\H^{n-1}_ e
\geq0.
\]
In particular, a function $u$ under the hypothesis of Theorem \ref{thm:main} is strictly nonlocal 2-convex and nonlocal semiconcave.

\medskip

Lastly, let us briefly comment on the notion of solution. In the local theory of $k$-Hessian equations, the emphasis has traditionally been on classical solutions \cite{Ca-Ni-Sp2,Ivochkina,Trudinger}, and subsequently on weak solutions. The notion of weak solution is a delicate issue, since any such notion must incorporate the admissibility of solutions, given that the operator $\sigma_k$ is elliptic only on the cone $\Gamma_k$.
 This has led to several frameworks, such as viscosity solutions with admissible ($k$-convex) test functions \cite{Trudinger-90,Urbas}, and the theory of Hessian measures \cite{TW1,TW2,TW3}; see also \cite{Trudinger-weak} and the survey~\cite{Wang-survey}.

In the nonlocal setting, viscosity solutions in the sense of Caffarelli--Silvestre~\cite{Caffarelli.Silvestre} constitute the natural
framework, see for instance \cite{FernandezReal-RosOton}. Indeed, it is well known that at any point where $u$ is touched from above by a $C^2$-function, the equation can be evaluated classically (see Lemma~\ref{lem:classicalsense} below), which makes working with viscosity and classical solutions somewhat similar.
In contrast with the local case, no admissibility restriction on test functions is needed in the nonlocal setting.
The nonlocal operators in this article and in \cite{Caffarelli-Charro,Caffarelli-Silvestre-16,Wu:thesis,Wu:paper} are defined as infima of
well-defined linear operators, and admissibility with respect to the class of coefficients is given by the equation itself through the finiteness of the infimum.

\subsection{Challenges} \label{sec:challenges}
Theorem~\ref{thm:main} is inspired by the nondegeneracy results by Caffarelli and Charro \cite{Caffarelli-Charro} for the fractional Monge-Amp\`ere operator $\mathcal D_s$ defined in \eqref{fractional.MA}. However, there is an important difference: 
For the Monge-Amp\`ere case, the determinant constraint $\det{A}=1$ allows more freedom in how the eigenvalues compensate for degeneracy in any one direction.
Degeneracy is essentially one-dimensional, and once one eigenvalue is small, the rest are free to be roughly equal.
On the other hand, the geometry of the set $\mathcal A_2$ is more complex, and this is a core point in our arguments.
 Degeneracy for $2\leq k<n$ involves multiple directions interacting nontrivially, and eigenvalues can degenerate at different rates that have to be tracked simultaneously. 
 This has a practical implication for eigenvalue estimates;
 it is well-known that for the Monge-Amp\`ere operator, lower and upper estimates are easily related through the structure of the equation. This is no longer the case for the 2-Hessian equation, where lower and upper bounds have to be treated ~independently.

Another technical challenge is that convexity is often assumed in the literature, even though it is not the natural hypothesis for $\sigma_k$ as it largely bypasses the central difficulty. Accordingly, in Theorem~\ref{thm:main} we do not require $u$ to be convex. 
Note that this is a common feature of several other results in fully nonlinear analysis, which are more readily established under a classical convexity hypothesis, rather than just $k$-convexity. An example of this is the aforementioned Christoffel-Minkowski problem, where the crucial $C^2$ estimate follows by assuming convexity, see \cite{Guan-Ren-Wang}, which can be relaxed to $k$-convexity only when $k=2$.
A related question where convexity simplifies the analysis is in the derivation of geometric inequalities for quermassintegrals. With considerable effort, Aleksandrov-Fenchel/isoperimetric inequalities have been obtained with just a $(k+1)$-convexity hypothesis using optimal transport methods \cite{Chang-Wang,Chang-Wang:isoperimetric}. Nevertheless, this assumption cannot be removed in their work.

Strict ellipticity for nonlocal $k$-Hessian operators when $k\geq2$ has been previously attempted in \cite{Wu:thesis,Wu:paper} ($k=2$) and \cite{Gan-Jiao} ($k\geq2$, under a convexity assumption).
Roughly speaking, these works adapt the strategy developed in \cite{Caffarelli-Charro}
for the nonlocal Monge-Amp\`ere operator ($k=n$), which proceeds in three steps: 
\begin{itemize}
 \item[\emph{(i)}] Establish strict positivity and boundedness of appropriate lower-dimensional fractional Laplacians of the solution (in \cite{Caffarelli-Charro}, one-dimensional fractional Laplacians); 
 \item[\emph{(ii)}] Decompose the $n$-dimensional integral in the definition of the nonlocal $k$-Hessian operator in terms of the lower-dimensional fractional Laplacians from step \emph{(i)} (in \cite{Caffarelli-Charro} this is done using polar coordinates); 
 \item[\emph{(iii)}] Combine the previous steps to show that degenerate coefficient matrices produce arbitrarily large $n$-dimensional integrals that do not contribute towards the infimum in the nonlocal $k$-Hessian.
 \end{itemize}

As it turns out, for the $2$-Hessian operator, the natural objects in step \emph{(i)} are fractional Laplacians restricted to hyperplanes, or $(n-1)$-dimensional subspaces.
If one is to follow the above program closely, one can replace the standard polar coordinates, used in step \emph{(ii)} in the Monge-Amp\`ere case, by an integral decomposition where hyperplanes take the role of lines, see for instance \cite[Formula~3.3]{Solmon} or \cite[Theorem~3.6]{IGT}.
 This introduces a new major difficulty since ``tilted" subspaces require more precise estimates to control their contribution to the total integral. This approach seems to  fail unless one assumes convexity.

In this paper, we introduce a novel approach that, rather than looking at each separate eigenvalue, involves a global geometric description that allows us to prove step~\emph{(iii)} directly, and bypass steps \emph{(i)} and \emph{(ii)} altogether; see Section~\ref{sec:keyest}. A key tool is a new representation formula for the fractional 2-Hessian operator, introduced in Lemma~\ref{lem:newrep} below. 

 Next, we present a more detailed overview of the ideas involved in the proof.

\subsection{Strategy of the proof of Theorem \ref{thm:main}.}
To establish the uniform ellipticity, we need to prove that the degenerate matrices do not contribute to the infimum in the definition of the nonlocal 2-Hessian operator \eqref{frac-2-hessian}. 
In the local setting (i.e., for $\sigma_2(D^2 u)$), the proof hinges on the following estimate, proved in Section~\ref{sec:idealocal},
\begin{equation}\label{key-estimate-intro}
\mathcal{L}_A u(x)=\tr(A D^2 u(x) A) \geq \gamma_0^{\rm loc} \tr(A^2), \quad \text{ for all } A\in \A_2, 
\end{equation}
for all $x\in \Omega$, and for some constant $\gamma_0^{\rm loc}>0$.
The trace on the right-hand side of \eqref{key-estimate-intro} is the square of the Frobenius norm of the matrix $A$, i.e.,
\[
\|A\|_F:=\sqrt{\tr(A^2)},
\]
a natural way to measure the size of coefficient matrices in this context.

We also observe that, under a strict convexity assumption $D^2u(x)\ge \gamma_0^{\rm loc} I$, estimate \eqref{key-estimate-intro} follows trivially. 
In our case, the main difficulty is that we only know that $D^2u(x)$ belongs to the cone $\Gamma_2$, so it could have negative eigenvalues.
Nevertheless, the possible degeneracies of the matrices in $\mathcal A_2$ are rigid enough that the contribution of the negative directions is compensated by the remaining~ones.

Thus, we need to understand the geometry of the set $\mathcal A_2$ and characterize its degeneracies.
The proof follows by identifying the set $\mathcal A_2$ as a hypersurface in $ \mathbb R^n$. For this, let us
denote the eigenvalues of 
$A$ and $A^2$, written in increasing order, by 
$$\lambda(A)=(\lambda_1,\hdots,\lambda_n)\quad \text{ and } \quad \mu(A):= \lambda(A^2)=(\mu_1,\hdots, \mu_n),$$
respectively, where $\mu_i=\lambda_i^2$, $i=1,\ldots,n$. Lemma \ref{lem:char} states that a matrix $A\in \S^n_+$ belongs to the class $\A_2$ if and only if 
 \begin{equation} \label{intro.eq1.hyperboloid}
 \sigma_2(\mu) - \frac{n-2}{2(n-1)}\sigma_1(\mu)^2=1.
 \end{equation} 
This is a two-sheet hyperboloid in the variable $\mu\in\mathbb R^n$, which is symmetric about the axis $\{ t (1, \hdots, 1) : t\in \R\}$, and with vertex at $t_*(1,\ldots,1)$ with $t_*$ defined in \eqref{definicion-t*}, see Figure \ref{fig:hyperboloid}. We define $A_*$ to be the ``vertex'' matrix, i.e, 
\begin{equation}\label{vertex-matrix-intro}
A_*:=\sqrt{t_*}I.
\end{equation}
\begin{figure}[t]
 \centering
 \includegraphics[width=0.6\linewidth]{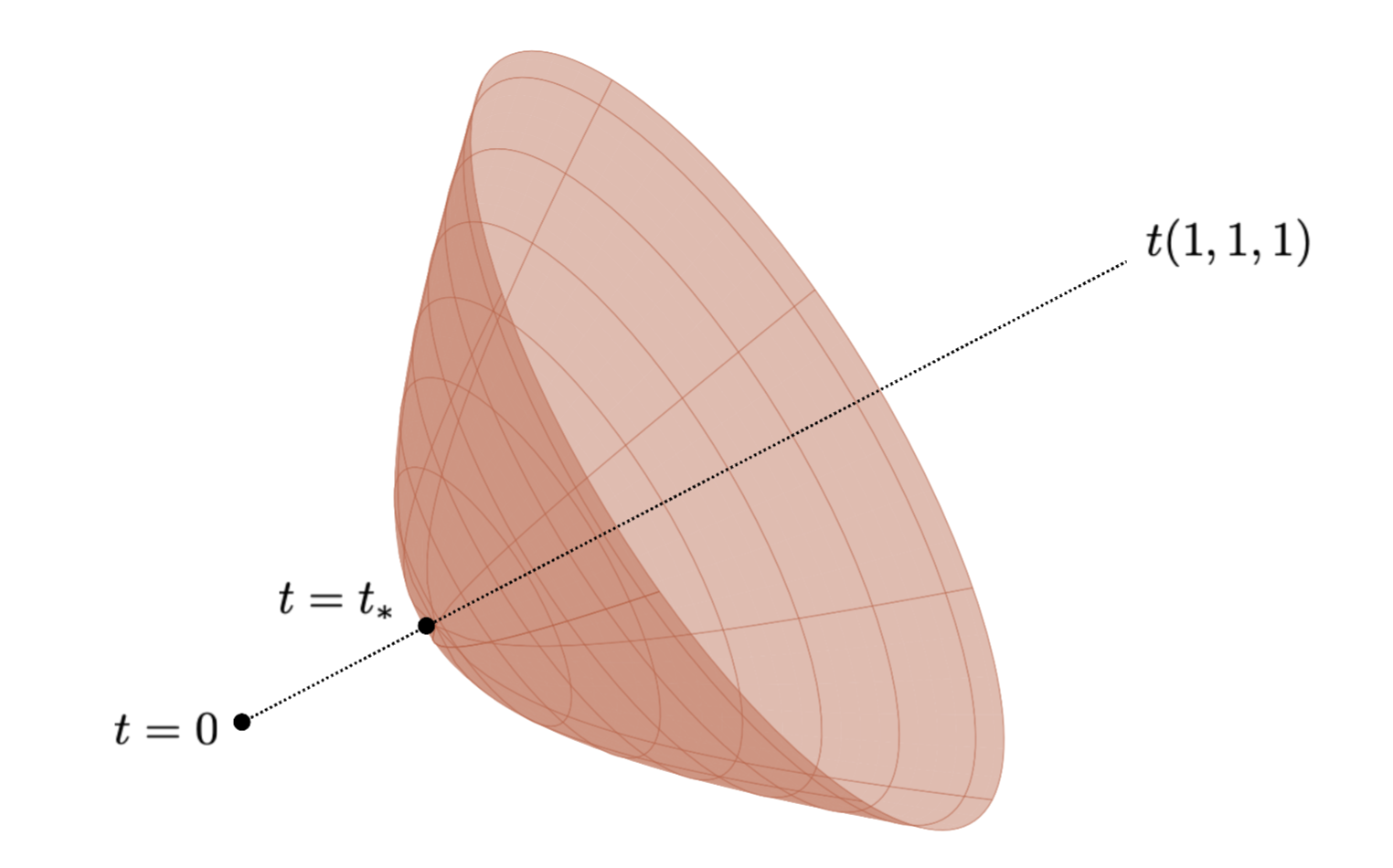}
 \caption{Positive sheet of the hyperboloid \eqref{intro.eq1.hyperboloid} with vertex at $t_*$.}
 \label{fig:hyperboloid}
 \end{figure}
 We also set
\begin{equation}\label{intro-m*}
m_*:=\|A_*\|_F^2=\sqrt{2n(n-1)}.
\end{equation}

 To describe the possible degeneracies for matrices $A$ in $\mathcal A_2$, observe that $\|A\|_F\geq \sqrt{m_*}$, and normalize
$$
\widetilde A :=\frac{1}{\|A\|_F}A.
$$
We denote by $\widetilde\mu=(\widetilde \mu_1, \ldots, \widetilde \mu_n)$ the eigenvalues of $\widetilde A^2$.
Then, the limit eigenvalues $\overline \mu=(\overline \mu_1,\ldots,\overline \mu_n)$ arising as $\|A\|_F\to\infty$ satisfy 
\begin{equation*} 
\frac{1}{n-1}\sigma_1(\overline \mu)^2 = |\overline\mu|^2,
\end{equation*}
which is a cone asymptotic to the hyperboloid \eqref{intro.eq1.hyperboloid} defined by $\mathcal A_2$.

First note that, if we look at a normal section of the hyperboloid, then $\widetilde \mu$ belongs to an $(n-2)$-dimensional sphere inside this normal section, whose distance to the normal section of the cone is controlled by the Frobenius norm $\|A\|_F$, see Figure \ref{fig:projection} and Lemma \ref{lem:sections}.

Unfortunately, $\widetilde A$ does not belong to the class $\mathcal A_2$, so it is not a valid test matrix in the infimum \eqref{frac-2-hessian}. Thus, we construct a perturbed matrix $A_\delta\in \mathcal A_2$ that belongs to a smooth arc on the hyperboloid, joining $A$ to the vertex $A_{*}$ (see Figure~\ref{fig.lateral.section}). Observe that this construction does not depend on which particular eigenvalue degenerates, as it is based on the norm $\|A\|_F$ alone.

Looking at the precise formula for $A_\delta$ from its definition \eqref{def.tau.delta}, we observe that it has two components: one in the cone, and one along the axis. Equivalently, one may use the formulation in \eqref{def.B.delta.tildeA} to decompose $A_{\delta}^2=\tau'\,\widetilde{A}^2+\delta'\, I$.

Now, in the local setting, the trace operator is linear, and one is able to separate the axial component from the one along the asymptotic cone, yielding
\begin{equation}\label{splitting-intro}
\begin{split}
\eta_0\leq\tr\big(A_{\delta}D^2u(x)A_{\delta}\big)
&
=\tau'\tr\big(\widetilde{A}D^2u(x)\widetilde{A}\big)+\delta'\Delta u(x).
\end{split} 
\end{equation}
To handle the Laplacian term above, one can use semiconcavity with constant $K$. Then, if $\delta'$ is small enough, we can get a bound from below for $\tr(\widetilde A D ^2 u(x)\widetilde A)$, from where we immediately obtain the key estimate \eqref{key-estimate-intro} by homogeneity.

The next step is to find the nonlocal analogue of \eqref{key-estimate-intro}. We observe that the trace operator in \eqref{key-estimate-intro} can be replaced by the average of second-order directional derivatives  on the sphere (see Lemma \ref{lem:average}). In particular, for $\I u$ as in \eqref{def.I.local}, it holds that
\begin{align*}
\frac{1}{n} \tr(A^t D^2 u(x) A) = \fint_{\mathbb{S}^{n-1}} \I u(x,Ae) \,d\H^{n-1}_e,
\end{align*}
see \eqref{representation.local} below.
In the nonlocal setting, one also has a similar representation formula for the nonlocal Hessian form \eqref{def.I.s.intro} (see Lemma~\ref{lem:newrep}).
We aim to prove that
\begin{equation}\label{intro.key.est.nonlocal}
\mathcal{L}_A^s u(x)= \frac{C_{n,s}}{4(1-s)} 
\int_{\mathbb{S}^{n-1}} \I^s u (x, Ae)\, d\H^{n-1}_e
\geq \gamma_0 \|A\|_F^{2s},
\end{equation}
for all $A\in \A_2,$ which replaces the local estimate \eqref{key-estimate-intro}, and 
is the key step (see Theorem \ref{thm:keyest}).

In this process, we need to control the $\delta' I$ term in $A_\delta$, which was quite straightforward in the local case \eqref{splitting-intro} due to the linearity of the trace in the coefficient matrix. In the nonlocal case, no such linearity property holds. Instead, we develop a novel comparison estimate that provides quantitative control of the difference between $\mathcal L^s_{\widetilde A}$ and the perturbed $\mathcal L^s_{A_\delta}$. More precisely, we will show in Proposition \ref{prop:tech-lem} that for general matrices $Q_1,Q_2\in\mathcal S^n_+$ such that $Q_2^2> Q_1^2$, it holds
\[
 \mathcal{L}_{Q_2}^su(x)- \mathcal{L}_{Q_1}^su(x) \leq c_{K,s} \left(\tr(Q_2^2-Q_1^2)\right)^s,
\]
where 
\begin{equation}\label{defi-C-0.intro}
c_{K,s}:=K \cdot\frac{\omega_{n-1}}{4 n^s}\cdot\frac{C_{n,s}}{1-s},
\end{equation}
with $c_{K,s}\to K$ as $s\to 1$.
Finally, we show that all our results are stable under the limit $s\to 1$, recovering the exact local results.

To the best of our knowledge, the above strategy is new, even in the local case. Our arguments, in particular, yield the existence of a constant $\Lambda(n,K,\eta_0)\geq 1$ such that
\begin{equation} \label{eq:localUE2}
 2\sigma_2(D^2 u(x) )^\frac{1}{2}= \min_{A\in \A_2\cap\mathcal{N}(\Lambda)} \tr(A D^2 u(x) A)
 \end{equation}
(see Section~\ref{sec:idealocal} below on the local case). Nevertheless, in the local case, this characterization
can also be obtained by classical linearization arguments, since $2\sigma_2^{1/2}$ is concave and homogeneous of degree one on
$\Gamma_2$.
In this way, the uniform ellipticity~\eqref{eq:localUE2} is equivalent to two-sided bounds on the linearization coefficients $\sigma_{1,i}(\lambda)$.
Estimates of this type for admissible solutions are classical in the $k$-Hessian literature, see \cite{Garding,Ca-Ni-Sp2,Trudinger,Wang-survey}. 
In this paper, we work directly with the Bellman representation~\eqref{2-Hessian} and the coefficient set $\A_2$, which is what ultimately allows the extension to the nonlocal setting (where no linearization is possible), and provides estimates that are preserved in the limit from nonlocal to local.

\subsection{Further remarks and open problems} \label{sec:openprob}

As stated in subsection~\ref{subsection.discussion.hyp}, condition~\eqref{NSC} can be seen as a nonlocal semiconcavity condition.
Naturally, the notion may also suggest defining a nonlocal convex function as one that satisfies
\begin{equation}\label{NCX}
\mathcal Q^s u(x;y,z) \geq 0, \qquad\textrm{$y,z\in\mathbb{R}^n$},
\end{equation}
whenever $s>1/2$, which, to the best of our knowledge, has not been considered in the literature.
There is, however, the related notion of fractional convexity of \cite{DelPezzo-Quaas-Rossi}, which is defined through the one-dimensional fractional Laplacians,  a
line-by-line condition where \eqref{NCX} is replaced by 
\begin{equation}\label{FCX}
\Theta^su(x,e)\geq0,\qquad e\in\mathbb{S}^{n-1}.
\end{equation}
This is the nonlocal counterpart of $\lambda_{\min}(D^2u)\geq0$, which, in the local case, is equivalent to~the~convexity of the quadratic form $\I u(x,\cdot)$, since a quadratic form is convex if and only if it is nonnegative.

\medskip
Once conditions for uniform ellipticity for the nonlocal 2-Hessian operator are established, a natural question to consider is the existence and uniqueness of solutions.
The nonlocal Monge-Amp\`ere and Hessian literature (see for instance \cite{Caffarelli-Charro,Caffarelli-Silvestre-16,Jhaveri-Stinga,Wu:thesis,Wu:paper}) has extensively developed a ``case study'' problem in all of $\mathbb{R}^n$ where solutions approach a given profile at infinity, which allows focusing on the behavior of the newly defined nonlocal operators and deferring the technical issues associated with boundary data.
In our context, the corresponding problem reads
\begin{equation}\label{test.problem}
\left\{
\begin{aligned}
& \K^s u=u-\phi&&\textrm{in\ }\mathbb{R}^n,\\
& u-\phi\to0 &&\textrm{as}\ |x|\to\infty,
 \end{aligned}
 \right.
\end{equation}
for a given smooth, strictly convex function $\phi : \mathbb{R}^n\to \mathbb{R}$ asymptotic to a cone, which forces $1/2<s<1$ for the operator to be well defined.
 Although we will not provide the details, one could prove the existence and uniqueness of solutions to \eqref{test.problem} closely following \cite{Caffarelli-Charro}.
 More specifically, both the uniqueness of solutions and existence via the Perron method follow from the comparison principle argument in \cite{Caffarelli-Charro}, which applies to general degenerate Bellman equations.
Existence for problems on bounded domains with a prescribed right-hand side $f$, and for the full range $s\in(0,1)$, is a nontrivial challenge (it remains largely open even for the nonlocal Monge-Amp\`ere operator), and the geometry of the boundary is expected to play a major role, as in the local case; see \cite{Ca-Ni-Sp2} for a discussion of domain admissibility.

\medskip
Finally, a further interesting open problem is to establish the uniform ellipticity of $k$-Hessian equations for $k\geq 3$. Strict ellipticity was proved in \cite{Gan-Jiao}, under the assumption that solutions are convex, which, as mentioned in Section~\ref{sec:challenges}, is not the natural admissibility condition for these equations. The techniques developed here appear to extend to the case $k\geq 3$ and allow us to remove the convexity assumption. We are currently investigating this extension in an ongoing~work.

\subsection{Organization of the paper}
The rest of the paper is organized as follows. In Section~\ref{sec:prelim}, we introduce the notation and some preliminary results. 
In Section~\ref{sec:A2class}, we characterize the class $\A_2$, and study the possible degeneracies, see Lemma~\ref{lem:char} and Corollary~\ref{cor:degeneracies}, respectively.
In Section~\ref{sec:idealocal}, we discuss our novel geometric approach to establish the uniform ellipticity for the local 2-Hessian equation. 
Section~\ref{app:tech-lem} is devoted to proving a comparison between nonlocal operators (not to be confused with a comparison principle) that will be used to extend the local proof to the nonlocal case.
Section~\ref{sec:keyest} is devoted to the nonlocal counterpart of said geometric approach; we prove the key estimate \eqref{intro.key.est.nonlocal}, see Theorem~\ref{thm:keyest}, and our main result, Theorem~\ref{thm:main}. 
Sections \ref{sec:NSC.discussion} and \ref{sec:counterexample} provide further insights on the nonlocal semiconcavity notion \eqref{NSC}, and a counterexample that 
justifies the need of such an assumption. Appendix \ref{appendix} states some probability results used in Section~\ref{sec:counterexample}.

\medskip

\section{Preliminaries}\label{sec:prelim}

In this section, we introduce the notation and recall some definitions and basic results.

\subsection{Notation}
Throughout the paper, $\S^n$ denotes the space of symmetric $n\times n$ real matrices, and $\S^n_+\subset\S^n$ the cone of symmetric positive 
definite matrices. Inequalities between symmetric matrices are understood in the sense that $A\leq B$ whenever $B-A\geq0$. Given $A\in\S^n$, we write 
$\lambda(A)=(\lambda_1,\dots,\lambda_n)$ for its eigenvalues, counted with multiplicity and arranged in increasing order, and, when $A\in\S^n_+$, we 
denote $\mu(A)=\lambda(A^2)$, so that $\mu_i=\lambda_i^2$ for $i=1,\ldots,n$. We denote by $\|A\|_F:=\sqrt{\tr(A^2)}$ the Frobenius 
norm of $A$, and by $\|A\|=\max_i|\lambda_i(A)|$ its spectral norm, which for symmetric matrices coincides with the operator norm induced by 
the Euclidean norm. For $\lambda\in\Rn$, we write $\sigma_k(\lambda)$ for the $k$-th elementary symmetric polynomial, and, with a slight abuse of notation, $\sigma_k(A):=\sigma_k(\lambda(A))$ for $A\in\S^n$.

We denote by $\mathbb S^{n-1}:=\{x\in\Rn:|x|=1\}$ the unit sphere of $\Rn$. By $\H^{n-1}$, we denote the $(n-1)$-dimensional Hausdorff measure, and 
indicate the integration variable as a subscript. The letters $e$ and $\theta$ are reserved for unit vectors on $\mathbb{S}^{n-1}$. We set $\omega_{n-1}:=\H^{n-1}(\mathbb S^{n-1})$, the surface measure of the unit sphere, not to be confused with the weight $\omega_s(y)=(1+|y|^{n+2s})^{-1}$ from the definition of $L^1_{\omega_s}(\Rn)$. As usual, $B_r(x)\subset\R^n$ denotes the open ball of radius $r>0$ centered at~$x\in\Rn$.

\subsection{The fractional Laplacian}
For $x, y\in \R^n$, we define the centered second difference of $u$ as
$$
\delta u(x,y) := u(x+y)+u(x-y)-2u(x). 
$$
We recall the weight $\omega_s(y):=\big(1+|y|^{n+2s}\big)^{-1}$, and say that $u\in L^1_{\omega_s}(\Rn)$ if 
$$
\| u\|_{L^1_{\omega_s}(\Rn)} := \int_{\Rn} \frac{|u(y)|}{1+|y|^{n+2s}}\, dy<\infty.
$$
Moreover, we recall the definition of the fractional Laplacian of $u\in C(\Rn)\cap L^1_{\omega_s}(\Rn)$ given by
\begin{equation}\label{defi-frac-lap}
\Delta^s u(x) := C_{n,s}\, {\rm P.V.} \int_{\Rn} \frac{u(y)-u(x)}{|y-x|^{n+2s}}\, dy
= \frac{C_{n,s}}{2} \int_{\Rn} \frac{\delta u (x,y)}{|y|^{n+2s}}\, dy,
\end{equation}
with $C_{n,s}$ defined in \eqref{eq:ctfraclapl}.
Strictly speaking, our definition of the fractional Laplacian defines $-(-\Delta)^s$, but for notational simplicity, we will omit the minus signs.

In the sequel, we will use the asymptotic behavior of the normalizing constant $C_{n,s}$ defined in~\eqref{eq:ctfraclapl}.
We have
\begin{equation}\label{eq:Cns.limit}
\lim_{s\to1}\frac{C_{n,s}}{1-s}=\frac{4\,\Gamma\big(\tfrac n2+1\big)}{\pi^{n/2}}=\frac{4n}{\omega_{n-1}}.
\end{equation}
 In particular, $C_{n,s}=O(1-s)$, as $s\to1$.

\subsection{Viscosity solutions} 
 We recall from \cite{Caffarelli.Silvestre} the notion of  viscosity solutions. 

\begin{defn} \label{def:viscosity}
Let $s\in (0,1)$, let $\Omega$ be an open set in $\Rn$, and let $f\in C(\Omega)$.
A function $u \in L^1_{\omega_s}(\Rn)$, upper (resp.~lower) semicontinuous in $ \Omega$, is a viscosity subsolution (resp.~supersolution) to $\K^su = f$ in $\Omega$, and we write 
$$\K^su\geq f \quad \text{in } \Omega \qquad (\textrm{resp. }\K^s u \leq f \ \text{ in } \Omega),$$ 
if every time all the following happen,
\begin{enumerate}[\it (i)]
\item $x$ is a point in $\Omega$,
\item $N$ is an open neighborhood of $x$ in $\Omega$,
\item $\psi$ is some $C^2$-function in $N$,
\item $\psi(x) = u(x)$,
\item $\psi(y) > u(y)$ (resp.~$\psi(y) < u(y)$) for every $y \in N \setminus \{x\}$,
\end{enumerate}
 and if we let \[ v := \begin{cases}
 \psi &\text{in } N \\
	 u &\text{in } \mathbb{R}^n \setminus N,
 \end{cases} \]
then we have $\K^sv(x) \geq f(x)$ (resp. $\K^sv(x) \leq f(x)$).
We say that $u\in C(\Rn)\cap L^1_{\omega_s}(\Rn)$ is a viscosity solution to $\K^su = f$ in $\Omega$ if it is both a viscosity subsolution and supersolution.
\end{defn}

\begin{rem}
    Given $u$ as in Definition \ref{def:viscosity}, whenever there exists $\psi$ satisfying \emph{(i)--(v)},  we say that $u$ can be \emph{touched from 
above} by a $C^2$-function at $x$ (resp. from below).
\end{rem}

The following lemma states that $\K^s u$ can be evaluated classically at those points where $u$ can be touched by a $C^2$-function;
see \cite[Lemma 3.3]{Caffarelli.Silvestre}.

\begin{lem}\label{lem:classicalsense}
Let $u \in L^1_{\omega_s}(\Rn)$ be upper semicontinuous in $ \Omega$.
If $u$ satisfies $\K^su\geq f$ in $\Omega$, in the viscosity sense, and $\psi\in C^2(N)$ touches $u$ from above at $x$, then 
$\K^su(x)\geq f(x)$ is defined in the classical sense.
The analogous statement holds for viscosity supersolutions, i.e., $\K^su\leq f$ in $\Omega$.
\end{lem}

\subsection{The one-dimensional fractional Laplacian and its properties}
Next, we collect some useful properties of the one-dimensional fractional Laplacians $\Theta^su$ defined in \eqref{definition-h-s}, as well as its homogeneous extension of degree $2s$, denoted $\mathcal I^su$, and defined in \eqref{def.I.s.intro}.
Observe that one can also write
\begin{equation}\label{def.alternate.I.s}
\I^s u(x,y)
=
\left\{
\begin{aligned}
&|y|^{2s}\,\Theta^su\Big(x,\tfrac{y}{|y|}\Big), &&\textrm{for}\ y\in \mathbb{R}^{n}\setminus\{0\},\\
&0, &&\textrm{for}\ y=0.
\end{aligned}
\right.
\end{equation}
In particular,
\begin{equation}\label{I.s-homogeneous}
\I^su(x,\rho y)=\rho^{2s} \I^su(x,y), \quad \text{ for all } \rho>0.
\end{equation}

In the next remark we collect two useful properties of $\I^su$.
\begin{rem} 
Let $e\in \mathbb{S}^{n-1}$ and $x,y\in\mathbb{R}^n$. Then,
\smallskip
\begin{enumerate}\itemsep5pt

\item[\emph{(i)}] 
$\displaystyle \I^su(x,y) = 2(1-s) \int_0^\infty \frac{\delta u(x,r y)}{r^{1+2s}}\, dr,$\quad and \quad
$\displaystyle \Theta^su(x,e) = 2(1-s) \int_0^\infty\frac{\delta u(x,r e)}{r^{1+2s}}\, dr.$

\item[\emph{(ii)}]
Let $x\in\Rn$ be such that $\I^s u(x,\cdot)\in L^1(\mathbb{S}^{n-1})$, see Remark \ref{remark.I.s.in.L1},
and assume
\eqref{NSC} holds for a.e.\ $y,z\in\Rn$. Then, 
\begin{equation}\label{bound.I.s}
\I^su(x,z)\leq K|z|^{2s}
\qquad\textrm{for a.e.}\ z\in\Rn,
\end{equation}
which follows from \eqref{NSC} through a standard mollification argument.
In particular, $\Theta^su(x,e)\leq K$ for a.e.\ $e\in\mathbb S^{n-1}$.
\end{enumerate}
\end{rem}

The following representation formula will be useful. Recall that, for $A\in \S^n_+$, we defined
\begin{equation}\label{definicion-Ls}
\mathcal{L}^s_A u(x) = \frac{C_{n,s}}{2} \int_{\Rn} \frac{\delta u(x,Ay)}{|y|^{n+2s}}\, dy.
\end{equation}

\begin{lem}[Representation formula] \label{lem:newrep}
Let $A\in \S^n_+$. Then
\[
\begin{split}
\mathcal{L}^s_A u(x)
&=\frac{C_{n,s}}{4(1-s)}\int_{\mathbb S^{n-1}}\I^su\big(x,\,A e \big)\,d\H^{n-1}_ e \\
&=\frac{C_{n,s}}{4(1-s)}\int_{\mathbb{S}^{n-1}} |A e |^{2s} \Theta^s u \Big(x, \frac{A e }{|A e |}\Big)\, d\H^{n-1}_ e.
\end{split}
\]
In particular, 
\begin{equation}\label{rep.formula.Ks}
\K^s u(x) = 
\frac{C_{n,s}}{4(1-s)}\,
\inf_{A\in \A_2}\int_{\mathbb S^{n-1}}\I^su\big(x,\,A e \big)\,d\H^{n-1}_ e .
\end{equation}
\end{lem}

\begin{proof}
By the usual polar coordinates formula,
\begin{align*}
\int_{\Rn} \frac{\delta u(x,Ay)}{|y|^{n+2s}}\, dy
&= \int_{\mathbb{S}^{n-1}} \Big(\int_0^\infty \frac{\delta u(x,\rho A e )}{\rho^{1+2s}}\, d\rho \Big) \,d\H^{n-1}_ e .
\end{align*}
Making the change of variables $r= |A e | \rho$, we get
\[
 \int_0^\infty \frac{\delta u(x,\rho A e )}{\rho^{1+2s}}\, d\rho
= |A e |^{2s}\int_0^\infty \frac{\delta u(x,r \frac{A e }{|A e |})}{r^{1+2s}}\, dr
= |A e |^{2s} \frac{1}{2(1-s)}\Theta^s u\Big(x,\frac{A e }{|A e |}\Big).
\]
Using \eqref{def.alternate.I.s}, the proof is complete.
\end{proof}

\begin{lem}[Compatibility condition]
\label{lemma:compatibility} Assume the hypotheses of Theorem \ref{thm:main}, and let $x\in\Omega$ be such that $u$ can be touched from above by a $C^2$-function. It holds that \begin{equation}\label{compatibility-non-local}
\eta_0\leq\K^s u(x)\leq  K \cdot \frac{\omega_{n-1}}{4}\cdot \frac{C_{n,s}}{1-s}\cdot \Big(\frac{2(n-1)}{n}\Big)^{\frac{s}{2}}
= c_{K,s}\,m_*^s,
\end{equation}
where $c_{K,s}$ is defined in \eqref{defi-C-0.intro} and $m_*$ in \eqref{intro-m*}.
Moreover, this inequality is stable under the limit $s\to 1$, and converges to the local compatibility condition in \eqref{calculation.compatibility.local} below.
\end{lem}

\begin{proof}
The hypotheses imply $\K^s u(x)\geq\eta_0$ holds classically, by Lemma~\ref{lem:classicalsense}.
Let $A_{*}=\sqrt{t_*}I\in \A_2$ be the vertex matrix defined in \eqref{vertex-matrix-intro}.
 Using the representation formula \eqref{rep.formula.Ks}, we get
\[
\begin{split}
 \eta_0 &\leq \K^s u(x)
 = \frac{C_{n,s}}{4(1-s)} \inf_{A\in \A_2} \int_{\mathbb S^{n-1}} \I^su\big(x,A e \big)\,d\H^{n-1}_ e \\
&\leq \frac{C_{n,s}}{4(1-s)} \int_{\mathbb S^{n-1}} \I^su\big(x,A_* e \big)\,d\H^{n-1}_ e 
=\frac{C_{n,s}}{4(1-s)}\,t_*^s \int_{\mathbb S^{n-1}} \I^su\big(x, e \big)\,d\H^{n-1}_ e.
\end{split}
\]
From here, using \eqref{bound.I.s}, we conclude
 \begin{equation*}
 \eta_0 \leq \K^s u(x)
 \leq K\,\frac{\omega_{n-1}}{4}\, \frac{C_{n,s}}{1-s}\, t_*^s= c_{K,s}\,m_*^s.
\end{equation*}
For the limit as $s\to 1$, just recall \eqref{eq:Cns.limit}.
\end{proof}

In the next result, we discuss the relation between hypotheses \eqref{eq:onesidedreg}-\eqref{eq:hyp.on.m} and \eqref{NSC}.
\begin{lem}\label{lemma:implication-semiconcavity}
Semiconcavity with a modulus, i.e., conditions \eqref{eq:onesidedreg}-\eqref{eq:hyp.on.m}, implies \eqref{NSC}. 
\end{lem}

\begin{proof}
Expanding the second differences,
\begin{equation}\label{expansion.second.differences}
\delta u(x,r(y+z))+\delta u(x,r(y-z))-2\,\delta u(x,ry)
=\delta u(x+ry,rz)+\delta u(x-ry,rz).
\end{equation}
Then, using \eqref{eq:onesidedreg}-\eqref{eq:hyp.on.m}, we calculate
\[
\begin{split}
\I^s u(x,y+z)&+ \I^s u(x,y-z) - 2\,\I^s u(x,y) \\
&= 2(1-s)\int_0^\infty \frac{\delta u(x,r(y+z))+\delta u(x,r(y-z))-2\delta u(x,ry)}{r^{1+2s}}\, dr\\
&= 2(1-s)\int_0^\infty \frac{\delta u(x+ry,rz)+\delta u(x-ry,rz)}{r^{1+2s}}\, dr\\
&\leq 4(1-s)\int_0^\infty \frac{m(r|z|)}{r^{1+2s}}\, dr 
= 4(1-s)\,|z|^{2s} \int_0^\infty \frac{m(r)}{r^{1+2s}}\, dr 
\leq 2K |z|^{2s},
\end{split}
\]
which is condition \eqref{NSC}.
\end{proof}


\section{The geometry and degeneracies of the class $\A_2$} \label{sec:A2class}

We prove a characterization of the matrices in $\A_2$ and some useful geometric properties. Given $A\in \S^n_+$, we denote the eigenvalues of $A$ and $A^2$ by 
$$\lambda(A)=(\lambda_1,\hdots,\lambda_n)\quad \text{ and } \quad \mu(A):= \lambda(A^2)=(\mu_1,\hdots, \mu_n),$$
respectively. Note that $\mu_i=\lambda_i^2$, for all $i=1,\hdots, n$.

\subsection{The class $\A_2$ as a hyperboloid} 
We identify the set $\mathcal A_2$ as a hypersurface in $ \mathbb R^n$.

\begin{lem}[Characterization] \label{lem:char}
 A matrix $A\in \S^n_+$ belongs to the class $\A_2$  if and only if 
 \begin{equation} \label{eq1}
 \sigma_2(\mu) - \frac{n-2}{2(n-1)}\sigma_1(\mu)^2=1, 
 \end{equation} 
or equivalently,
\begin{equation} \label{eq:classA2}
 \frac{1}{n-1}\sigma_1(\mu)^2-|\mu|^2=2.
 \end{equation}
\end{lem}

\begin{proof}
 Let $A\in \S^n_+$ be such that $\mu_i :=\lambda_i^2 = \sigma_{1,i}(\zeta)$ for some $\zeta\in \Gamma_2$. We will see that $\sigma_2(\zeta)=1$ if and only if $\mu$ satisfies $\sigma_2(\mu) - \frac{n-2}{2(n-1)}\sigma_1(\mu)^2=1$.
 Indeed, observe that
\[
\begin{split}
\sigma_2(\sigma_{1,i})
&=\sum_{1\leq i<j\leq n} \sigma_{1,i}\sigma_{1,j}
=
\sum_{1\leq i<j\leq n} (\sigma_{1}-\zeta_i)(\sigma_{1}-\zeta_j)\\
&=\sum_{1\leq i<j\leq n} \left\{\sigma_{1}^2-(\zeta_i+\zeta_j)\sigma_1+\zeta_i\zeta_j\right\}\\
&=\sigma_{1}^2 \,\frac{n(n-1)}{2}-\sum_{1\leq i<j\leq n} (\zeta_i+\zeta_j)\sigma_1+\sum_{1\leq i<j\leq n}\zeta_i\zeta_j\\
&=\sigma_{1}^2 \,\frac{n(n-1)}{2}-(n-1) \sigma_1\sigma_1+\sigma_2
=\sigma_{1}^2 \,\frac{(n-1)(n-2)}{2}+\sigma_2.
\end{split}
\]
Moreover,
\[
\big(\sigma_1(\sigma_{1,i})\big)^2
=\left((n-1) \sigma_1\right)^2=(n-1)^2 \sigma_1^2.
\]
Therefore, $\mu_i=\sigma_{1,i}(\zeta)$ satisfies
\[
\sigma_2(\mu)-\frac{n-2}{2(n-1)}\sigma_1(\mu)^2=\sigma_2(\zeta),
\]
and the first claim holds. Note also that 
$$
\sigma_1(\mu)^2 = \Big( \sum_{i=1}^n \mu_i \Big)^2 = \sum_{i=1}^n \mu_i^2 + 2 \sum_{i<j} \mu_i\mu_j = |\mu|^2 + 2 \sigma_2(\mu).
$$
Hence, $\sigma_2(\mu)= \frac{\sigma_1(\mu)^2-|\mu|^2}{2}$, and we can write \eqref{eq1} equivalently as
$$
 \frac{\sigma_1(\mu)^2-|\mu|^2}{2} - \frac{n-2}{2(n-1)}\sigma_1(\mu)^2=1,
$$
which simplifies to \eqref{eq:classA2}, 
as desired.
\end{proof}

\begin{rem}
 The hypersurface defined by \eqref{eq1} is a two-sheet hyperboloid, see Figure~\ref{fig:hyperboloid}, which is symmetric about the axis $\{ t (1, \hdots, 1) : t\in \R\}$. 
 In the positive sheet, the vertex, or lowest point along the axis, is attained at $\mu_*=t_*(1,\hdots, 1)$, with 
 \begin{equation}\label{definicion-t*}
 t_*:=\left(\frac{2(n-1)}{n}\right)^{1/2}.
 \end{equation}
At that point, the tangent plane is $\sigma_1(\mu)=m_*$, where we define
\begin{equation*}
m_*:=\sqrt{2n(n-1)}.
\end{equation*}
We denote by $A_*$ the ``vertex'' matrix, i.e., 
\begin{equation}\label{def.A.star}
 A_*=\sqrt{t_*}\,I
 =
 \left(\frac{2(n-1)}{n}\right)^{1/4}I.
\end{equation}
The matrix $A_*$ has eigenvalues given by $\mu_*^{1/2}$. It holds
 \begin{equation*}
 \|A_*\|_F:=\sqrt{\tr(A_*^2)}=\sqrt{m_*}.
 \end{equation*}
\end{rem}
 
The previous remark leads to a sharp lower estimate of the Frobenius norm
\[
\|A\|_F:=\sqrt{\tr(A^2)}
\]
of matrices $A\in \A_2$.
Namely, since $\sigma_1(\mu)=\tr(A^2)=\|A\|_F^2$, the minimum of $\|A\|_F$ over $\A_2$ must be attained where a level plane of $\sigma_1(\mu)$ becomes tangent to the hyperboloid \eqref{eq1}.

\begin{lem} \label{lem:trace}
 For every $A\in \A_2$, it holds that
 $$
 \|A\|_F^2=\tr(A^2) \geq m_*,
 $$
with equality if and only if $A=A_*$, the vertex.
\end{lem}

\begin{proof}
 Given $A\in \A_2$, by \eqref{eq:classA2} and the Cauchy-Schwarz inequality, we have
 $$
 \frac{1}{n-1} \sigma_1(\mu)^2 = 2 + |\mu|^2 \geq2+\frac{1}{n}\sigma_1(\mu)^2.
 $$ 
 Therefore,
 \[
\tr(A^2)^2=\sigma_1(\mu)^2 \geq 2n(n-1)=m_*^2.
 \]
 Equality in Cauchy-Schwarz holds if and only if $\mu=t(1,\dots,1)$, i.e., $A^2= t\,I$, and then \eqref{eq:classA2} makes $t=t_*$.
\end{proof}

\subsection{Eigenvalue estimates and possible degeneracies}
We provide eigenvalue estimates for matrices in $\A_2$, and study the possible degeneracy cases.

\begin{lem} \label{lema1}
 Let $A\in \A_2,$ with eigenvalues $\lambda_1\leq\lambda_2\leq\ldots\leq\lambda_n$. Then, 
 $$
 \lambda_i \lambda_j \geq 1\quad\textrm{for any $i\neq j$}.
 $$
In particular, if $\lambda_1<1$,
 it is the only eigenvalue smaller than 1, and hence simple.
\end{lem}

\begin{proof}
 Let $A\in \A_2$. Without loss of generality, let $i=1$ and $j=2$. For $\mu_1, \mu_2>0$ fixed, define the function
 $$
 G(\mu_3, \hdots, \mu_{n}) := \sigma_2(\mu)-\frac{n-2}{2(n-1)} \sigma_1(\mu)^2.
 $$
 To maximize $G$, we compute
 $$
 \frac{\partial G}{\partial \mu_i} = \sum_{j\neq i} \mu_j - \frac{n-2}{n-1} \sum_{j=1}^n \mu_j 
 = \frac{1}{n-1} \sum_{j\neq i} \mu_j - \frac{n-2}{n-1} \mu_i, \quad \text{ for all } 3 \leq i \leq n. 
 $$
 Setting $ \frac{\partial G}{\partial \mu_i}=0$, we see that
 \[
 \mu_ i = \frac{1}{n-2} \sum_{j \neq i} \mu_j\quad  \textrm{for all $3 \leq i \leq n,
 $}
 \]
 and, from there, that $\mu_3=\cdots=\mu_n =\mu_1+\mu_2$.
 To check that this is an absolute maximum, we will verify that $D^2 G \leq 0$.
 Note that $G$ is a quadratic polynomial, and
 $$
 \frac{\partial^2 G}{\partial \mu_i \partial \mu_\ell} = \frac{1}{n-1} (1-\delta_{i\ell}) - \frac{n-2}{n-1}\delta_{i\ell}.
 $$
 Hence, 
 $$
 D^2 G = \frac{1}{n-1} ( J_{n-2} - (n-1) I_{n-2}),
 $$
 where $I_{n-2}$ is the identity matrix and $J_{n-2}$ is the matrix with all entries equal to 1. Note that $J_{n-2}$ has two eigenvalues, $a=0$ (with multiplicity $n-3$) and $a=n-2$ (with multiplicity~1). Moreover, if $b$ is an eigenvalue of $D^2 G$, then
 $$
 b= \frac{a-(n-1)}{n-1}.
 $$
 Therefore, $D^2G$ has two eigenvalues, $b=-1$ (with multiplicity $n-3$) and $b=-\frac{1}{n-1}$ (with multiplicity 1), which implies that $D^2G\leq 0.$ 
 It follows that 
 $$
 G_{\max} = G(\mu_1+\mu_2,\hdots,\mu_1+\mu_2) = \mu_1 \mu_2.
 $$
 Since $G(\mu_3, \hdots, \mu_{n})=1$ for $A\in \A_2$, then
 $$
 1\leq \mu_1\mu_2=\lambda_1^2\lambda_2^2,
 $$
 from where we conclude that $\lambda_1\lambda_2\geq 1$.
\end{proof}

\begin{lem} \label{lema2}
 Let $A\in \A_2,$ with eigenvalues $\lambda_1\leq\lambda_2\leq\ldots\leq\lambda_n$. Then,
 $$
 1- \Big(\frac{\lambda_i}{\lambda_{n}}\Big)^2 \leq 2\, \frac{\lambda_{1}}{\lambda_{n}}, \quad \text{ for all } i=2,\ldots,n.
 $$
\end{lem}

\begin{proof}
 The case $i=n$ is trivial, so let us assume $2\leq i< n$. For $\mu_j=\lambda_j^2$, define the quantities
 $$
 \alpha:= \sum_{j=2}^n \mu_j, \quad \beta := \sum_{j=2}^n \mu_j^2, \quad \gamma := \frac{1}{n-1}\Big(\beta - \frac{1}{n-1} \alpha^2\Big).
 $$
 Note that $\gamma\geq 0$ by the Cauchy-Schwarz inequality. By \eqref{eq:classA2}, we have
 $$
 \frac{1}{n-1} (\mu_1 + \alpha)^2 - \mu_1^2 - \beta =2,
 $$
 which is equivalent to
 $$
 \beta-\frac{1}{n-1} \alpha^2 = \frac{2 \alpha \mu_1}{n-1} - \frac{n-2}{n-1}\mu_1^2 -2.
 $$
 Since $\alpha\leq (n-1)\mu_n$, and the other two terms are negative, it follows that
 \begin{equation*}
 (n-1) \gamma \leq \frac{2 \alpha \mu_1}{n-1} \leq 2 \mu_1 \mu_n.
 \end{equation*}
 On the other hand, for all $2\leq i< n$, we have 
 \begin{equation*} 
 (n-1)\gamma =\sum_{j=2}^n \Big(\mu_j - \frac{\alpha}{n-1}\Big)^2 \geq \Big(\mu_n - \frac{\alpha}{n-1}\Big)^2 + \Big(\mu_i - \frac{\alpha}{n-1}\Big)^2\geq \frac{1}{2}(\mu_n-\mu_i)^2,
 \end{equation*}
where in the last inequality, we have used that $\big(\frac{a+b}{2}\big)^2\leq \frac{a^2}{2}+\frac{b^2}{2}$ for all $a,b\in\mathbb{R}$.
 Hence, combining both inequalities, we get
 $$
 (\mu_n - \mu_i)^2 \leq 4 \mu_1 \mu_n.
 $$
 Dividing by $\mu_n^2$, it follows that
 $$
 \Big(1 - \frac{\mu_i}{\mu_n}\Big)^2 \leq 4 \Big( \frac{\mu_1}{\mu_n}\Big),
 $$
 which, taking the square root, gives the result.
\end{proof}

From the previous lemmas, we deduce the asymptotic behavior of the eigenvalues of the matrices $A$ in the class $\A_2$. There are only two possible degeneracies: either one eigenvalue is small and the rest are comparably large, or all eigenvalues are large.

\begin{cor}[Degeneracies in $\A_2$] \label{cor:degeneracies}
Let $\{A^{(k)}\}_{k\in\mathbb{N}}$ be a sequence of matrices in $\A_2$. For each fixed $k,$ let
$\lambda_1^{(k)}\leq\lambda_2^{(k)}\leq\ldots\leq\lambda_n^{(k)}$ be the eigenvalues of $A^{(k)}$.
Then, the following hold:
\begin{enumerate}[(i)]
 \item If
 $\lambda_1^{(k)} \to 0$ as $k\to\infty$, then $\lambda_j^{(k)} \to \infty$ as $k\to\infty$ for all $2\leq j\leq n$, all at the same rate, in the sense that ${\lambda_j^{(k)}}/{\lambda_{n}^{(k)}} \to 1$ as $k\to\infty$.

 \item If $\lambda_n^{(k)} \to \infty$ as $k\to\infty$, then: 
 
 \begin{enumerate}[(a)]
 \item 
 either $\lambda_{1}^{(k)}$ remains bounded and $\lambda_j^{(k)}\to\infty$ 
 as $k\to\infty$ for all $2\leq j\leq n$;
 \item or $\lambda_j^{(k)}\to \infty$ as $k\to\infty$ for all $j$.
 \end{enumerate} 
\end{enumerate}
\end{cor}

\begin{proof}
If $\lambda_1^{(k)}\to0$, Lemma~\ref{lema1} forces $\lambda_j^{(k)}\to\infty$ for all $2\leq j\leq n$, and then Lemma~\ref{lema2} gives ${\lambda_j^{(k)}}/{\lambda_{n}^{(k)}} \to 1$, proving $(i)$.
If $\lambda_n^{(k)}\to\infty$, as $k\to\infty$, and there is $C$ independent of $k$ such that $\lambda_{1}^{(k)}\leq\lambda_2^{(k)}\leq C$, then Lemma~\ref{lema2} gives 
\[
 1-\bigg(\frac{C}{\lambda_{n}^{(k)}}\bigg)^2\leq
 1- \bigg(\frac{\lambda_2^{(k)}}{\lambda_{n}^{(k)}}\bigg)^2 \leq \frac{2C}{\lambda_{n}^{(k)}},
\]
a contradiction.
Therefore, only $\lambda_{1}^{(k)}$ can stay bounded, which proves $(ii)$.
\end{proof}

One can also characterize the degeneracy cases for matrices in $\A_2$ in terms of the Frobenius~norm.
Given $\Lambda\geq 1$, we recall the set $\mathcal{N}(\Lambda)$, defined in \eqref{eq:nondeg}, of nondegenerate symmetric matrices with respect to $\Lambda$.

\begin{lem}
[Degeneracies in $\A_2$ in terms of the Frobenius norm]
\label{lem:frob.degen}
Let $\{A^{(k)}\}_{k\in\mathbb{N}}$ be a sequence of matrices in $\A_2$. For each fixed $k,$ let
$\lambda_1^{(k)}\leq\lambda_2^{(k)}\leq\ldots\leq\lambda_n^{(k)}$ be the eigenvalues of $A^{(k)}$. Then, the following are equivalent:
\vspace{+2pt}
\begin{enumerate}\itemsep3pt
 \item[(i)] $\displaystyle
\|A^{(k)}\|_F^2$ remains bounded as $k\to\infty.$
\item[(ii)] $\displaystyle\{A^{(k)}\}_{k\in\mathbb{N}}\subset\mathcal N(\Lambda)$ for some $\Lambda\geq1$.
\end{enumerate}
\vspace{+1pt}
Conversely, $\displaystyle\limsup_{k\to\infty}\|A^{(k)}\|_F^2=\infty$ if and only if $\{A^{(k)}\}_{k\in\mathbb{N}}$ has a subsequence degenerating in one of the ways described in Corollary~\ref{cor:degeneracies}.
\end{lem}

\begin{proof}
The result follows from the estimate
\begin{equation}\label{estima.for.est.coro.rates}
(\lambda_n^{(k)})^2\leq\tr((A^{(k)})^2)\leq n(\lambda_n^{(k)})^2.
\end{equation}
Indeed, if $\{A^{(k)}\}\subset\mathcal N(\Lambda)$, then trivially $\tr((A^{(k)})^2)\leq n\Lambda^2$.
On the other hand, assuming $\tr((A^{(k)})^2)\leq C$ independently of $k$, we have $\lambda_n^{(k)}\leq\sqrt C$. 
Then, by Lemma~\ref{lema1},
$$
\lambda_1^{(k)}\geq \frac{1}{\lambda_n^{(k)}}\geq\frac{1}{\sqrt C},
$$ 
and, therefore, $A^{(k)}\in\mathcal N(\sqrt C)$.
\end{proof}

\subsection{The asymptotic cone}
There is a nice geometric interpretation of the degenerate cases described above: 
The hyperboloid \eqref{eq1} is asymptotic to the cone
\begin{equation}\label{eq:asymp.cone}
\frac{1}{n-1}\,\sigma_1(\mu)^2=|\mu|^2,
\end{equation}
and Lemma~\ref{lem:frob.degen} allows us to quantify the degeneracy of matrices $A\in\A_2$ in terms of the distance between the cross-sections of the cone and the hyperboloid, appropriately normalized. 

More precisely, given $A\in\A_2$, let us normalize it by
\begin{equation}\label{def.tilde.A.mu}
\widetilde A :=\frac{1}{\|A\|_F}
A,\quad\textrm{or, equivalently,}\quad\widetilde{\mu}=\frac{\mu}{\sigma_1(\mu)},
\end{equation}
so that $\|\widetilde{A}\|_F^2=\tr(\widetilde A^2)=\sigma_1(\widetilde\mu)=1$. 
Moreover, the eigenvalues $\widetilde \mu\in (0,1]^n$ of $\widetilde A^2$ satisfy the normalized equation
\begin{equation}\label{normalized.eq.tilde.mu}
|\widetilde \mu|^2 = \frac{1}{n-1}- \frac{2}{\|{A}\|_F^4}.
\end{equation}

Both the hyperboloid and the cone are symmetric about the axis $\{ t (1, \hdots, 1) : t\in \R\}$,
and the level sets of $\sigma_1(\mu)$ are hyperplanes orthogonal to it.
Then, the normalization~\eqref{def.tilde.A.mu} can be seen as a central projection from the origin onto the hyperplane $\{\sigma_1(\mu)=1\}$.
It acts on each hyperplane  $\{\sigma_1(\mu)=t\}$ as the homothety  $\mu\mapsto\widetilde \mu=\mu/t$.

As it turns out, the normalized sections of the hyperboloid are $(n-2)$-dimensional spheres inside the hyperplane $\{\sigma_1(\mu)=1\}$, centered at
\begin{equation}\label{def.bar.c}
\overline{c}:=\big(1/n,\dots,1/n\big), 
\end{equation}
 and increasing to a fixed sphere given by the projection, or normalized section, of the cone (which is invariant under these homotheties). Let us denote the respective radii by
\begin{equation}\label{definition-r}
\widetilde r(t):=\Big(\frac{1}{n(n-1)}-\frac{2}{t^2}\Big)^{1/2}
\quad\ \textrm{and}\qquad
\overline r_\infty:=\Big(\frac{1}{n(n-1)}\Big)^{1/2},
\end{equation}
and observe that $\widetilde r(t)$ is increasing and converges to  $\overline r_\infty$ as $t\to\infty$, see Figure~\ref{fig:projection}.

\begin{figure}[t]
 \centering
 \includegraphics[width=0.7\linewidth]{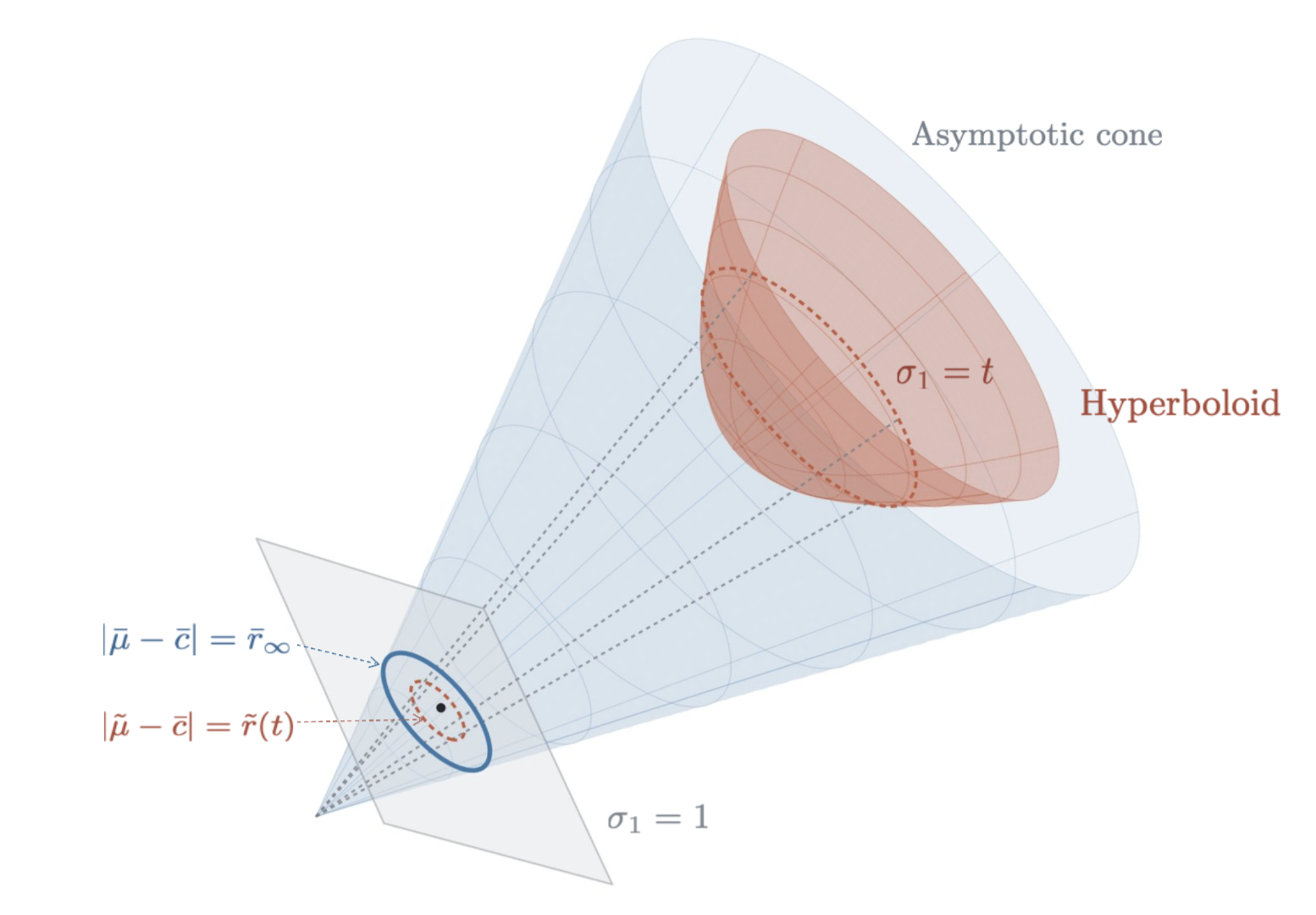}
 \caption{The asymptotic cone and the normalized sections.}
 \label{fig:projection}
 \end{figure}

We have the following.

\begin{lem}[Convergence of normalized sections to a limiting sphere]\label{lem:sections}
Given $A\in\A_2$, define $\widetilde{\mu}$ as in~\eqref{def.tilde.A.mu}. Then:
\begin{enumerate}\itemsep3pt
 \item[(i)] It holds that
\begin{equation}\label{norm.section.hyper}
\sigma_1(\widetilde\mu)=1, 
\qquad 
|\widetilde\mu-\overline c|=\widetilde{r}\big(\|A\|_F^2\big),
\end{equation}
with
\begin{equation}\label{eq:radialgap}
0\leq\overline r_\infty-\widetilde{r}\big(\|A\|_F^2\big)
\leq\frac{2\sqrt{n(n-1)}}{\|A\|_F^4}.
\end{equation}

\item[(ii)] In particular, if $\{A^{(k)}\}_{k\in\mathbb{N}}\subset\A_2$ is a degenerating sequence such that
$\|A^{(k)}\|_F\to\infty$, then the $\widetilde\mu^{(k)}$ accumulate on the $(n-2)$-dimensional sphere
\begin{equation}\label{norm.section.hyper.cone}
 \sigma_1(\overline\mu)=1,
 \qquad
|\overline\mu-\overline c|=\overline r_\infty,
\end{equation}
which is the normalized section of the asymptotic cone \eqref{eq:asymp.cone}. 
Equivalently, one can write the asymptotic sphere \eqref{norm.section.hyper.cone} as
\begin{equation}\label{alternative.asymp.esphere.def}
\sigma_1(\overline\mu)=1,\qquad
|\overline\mu|^2=\tfrac{1}{n-1}.
\end{equation}

\item[(iii)] 
Moreover, every point $\overline\mu$ in \eqref{norm.section.hyper.cone} satisfies $\overline\mu_i\in\big[0,\tfrac2n\big]$ for all $i=1,\ldots,n$; specifically, the limiting sphere is tangent to each coordinate hyperplane $\{\overline\mu_i=0\}$, and $\overline\mu$ has a zero coordinate if and only if it is one of the $n$ coordinate permutations of $\tfrac{1}{n-1}\,(0,1,\dots,1)$.
\end{enumerate}
\end{lem}

\begin{proof}
Using $\langle\widetilde\mu,\overline c\rangle=\tfrac1n\sigma_1(\widetilde\mu) =\tfrac1n=|\overline c|^2$ and \eqref{normalized.eq.tilde.mu}, we have
\[
|\widetilde\mu-\overline c|^2 = |\widetilde\mu|^2 + |\overline c|^2 - 2\langle\widetilde\mu,\overline c\rangle
= \frac{1}{n-1}- \frac{2}{\|A\|_F^4} -\frac1n =\widetilde{r}\big(\|A\|_F^2\big)^2,
\]
which yields \eqref{norm.section.hyper}. To prove \eqref{eq:radialgap}, notice that
\[
0\leq\overline r_\infty-\widetilde{r}\big(\|A\|_F^2\big)
=\frac{\overline r_\infty^2-\widetilde{r}\big(\|A\|_F^2\big)^2}{\overline r_\infty+\widetilde{r}\big(\|A\|_F^2\big)}
\leq\frac{2}{\overline r_\infty\,\|A\|_F^{4}}
=\frac{2\sqrt{n(n-1)}}{\|A\|_F^{4}},
\]
thus completing the proof of \emph{(i)}. 
Regarding \emph{(ii)}, for any converging subsequence $\widetilde\mu^{(k)}\to\overline \mu$ as $k\to\infty$, we can 
take limits in~\eqref{norm.section.hyper} using \eqref{eq:radialgap}, to show that $\overline \mu$ satisfies \eqref{norm.section.hyper.cone}, as desired.

Let us show \emph{(iii)}.  Clearly, any permutation of $\tfrac{1}{n-1}(0,1,\dots,1)$ satisfies \eqref{norm.section.hyper.cone}. To prove the converse,
assume $\overline\mu$ satisfies \eqref{norm.section.hyper.cone}, and observe that $\sigma_1(\overline\mu-\overline c)=0$. In particular,
\[
\overline\mu_i-\frac1n = \langle \overline\mu-\overline c , e_i-\overline c \rangle
\]
for all $i=1,\ldots,n,$ and the Cauchy--Schwarz inequality gives
\[
\Big|\overline\mu_i-\frac1n\Big|= |\langle \overline\mu-\overline c , e_i-\overline c \rangle|
\leq |\overline\mu-\overline c|\cdot|e_i-\overline c|
\leq\overline r_\infty\sqrt{\frac{n-1}{n}}=\frac1n.
\]
We deduce that $\overline\mu_i\in[0,\tfrac2n]$ for all $i=1,\ldots ,n$. Moreover, if $\overline\mu_j=0$ for some $j$, then equality holds in the Cauchy--Schwarz inequality, which forces $\overline\mu-\overline c = \frac{1}{n-1}(\overline c-e_j),$
and, from there
\[
\overline\mu=\frac{n}{n-1}\,\overline c-\frac{1}{n-1}\,e_j=\frac{1}{n-1}\sum_{i\neq j} e_i,
\]
which is the $j$-th coordinate permutation of $\tfrac{1}{n-1}(0,1,\dots,1)$, as desired.
In particular, for each $i$, the limiting sphere lies in the half-space $\{\overline\mu_i\geq0\}$ and meets the hyperplane $\{\overline\mu_i=0\}$ exactly at the $i$-th coordinate permutation of $\tfrac{1}{n-1}(0,1,\dots,1)$, which proves the tangency.
\end{proof}

\begin{rem}
The radius $\widetilde r(t)$ is well defined if and only if $t\geq m_*$, which, by Lemma~\ref{lem:trace}, is satisfied with $t=\|A\|_F^2$ for every $A\in\A_2$.
Thus the normalized sections in~\eqref{norm.section.hyper} are nonempty. Moreover, $\widetilde{r}(\|A\|_F^2)$ vanishes at the vertex matrix $A_*$,
the equality case of Lemma~\ref{lem:trace}.
For this matrix, the section of the hyperboloid reduces to the vertex $\mu_*=t_*\,(1,\dots,1)$ and, after normalization, to the point~$\overline c$.
\end{rem}

Finally, let us point out that Lemma \ref{lem:sections} provides a refinement of Corollary~\ref{cor:degeneracies}. 

\begin{cor}\label{coro.vanishing.rate}
For any $A\in\A_2$, we have the estimate
\begin{equation}\label{eigenvalues.converge.comparable}
\frac1n\bigg(\frac{\lambda_1}{\lambda_n}\bigg)^{2}
\leq \widetilde\mu_1 \leq \bigg(\frac{\lambda_1}{\lambda_n}\bigg)^{2}.
\end{equation}
 In particular, for any degenerating sequence $\{A^{(k)}\}_{k}\subset\A_2$ such that $\|A^{(k)}\|_F\to\infty$, 
the accumulation points of $\{\widetilde\mu^{(k)}\}$ have a vanishing first coordinate if and only if $\lambda_1^{(k)}/\lambda_n^{(k)}\to0$, and, in that case,
 $\widetilde\mu^{(k)}\to\tfrac{1}{n-1}\,(0,1,\dots,1)$.
\end{cor}

\begin{proof}
 Estimate \eqref{eigenvalues.converge.comparable} follows easily from 
\eqref{estima.for.est.coro.rates}. 
\end{proof}

According to Corollary \ref{coro.vanishing.rate}, $\widetilde\mu_1^{(k)}\to0$ as $k\to\infty$ whenever the smallest eigenvalue $\lambda_1^{(k)}$ diverges more slowly than the largest one, $\lambda_n^{(k)}$, which extends the cases of a vanishing or bounded first eigenvalue, considered in Corollary \ref{cor:degeneracies}. 
Conversely, \eqref{eigenvalues.converge.comparable} also implies that accumulation points $\overline\mu$ with all nonzero entries must occur when all the eigenvalues of $A^{(k)}$ diverge at a comparable rate.

\section{The idea in the local case} \label{sec:idealocal}

To illustrate the strategy of the proof of our main result, Theorem \ref{thm:main}, we first explain the ideas in the local setting $s=1$, where they are more transparent.
In the local case, semiconcavity plays the role of the nonlocal semiconcavity hypothesis \eqref{NSC}, so let us assume $u \in C^2(\Omega)$ is a semiconcave function with constant $K$, i.e., such that $D^2u(x) \leq KI$ for all $x\in \Omega$. 
Furthermore, assume that for some $\eta_0>0$,
 \begin{equation} \label{eq:assumloc}
 2\sigma_2(D^2 u(x))^\frac{1}{2} = \inf_{A\in \A_2} \tr(A D^2 u(x) A)
 \geq \eta_0, \quad \text{ for all } x\in \Omega.
 \end{equation}
 Our goal is to establish that degenerate matrices do not count for the infimum,
 that is, we intend to show that there is a constant $\Lambda\geq 1$, depending only on $n$, $K$, and $\eta_0$, such that
 \begin{equation} \label{eq:localUE}
 2\sigma_2(D^2 u(x) )^\frac{1}{2}= \inf_{A\in \A_2\cap\mathcal{N}(\Lambda)} \tr(A D^2 u(x) A),
 \end{equation}
 where $\mathcal{N}(\Lambda)$ denotes the class of nondegenerate matrices introduced in \eqref{eq:nondeg}.

 \begin{rem}\label{remark.compatibility.local}
 The hypotheses force a compatibility condition of the form $\eta_0\leq C(n) K$ among the data.
 More precisely, taking $A_{*}$ as in \eqref{def.A.star}, and using the semiconcavity of $u$, 
\begin{equation}\label{calculation.compatibility.local}
 \begin{split}
 \eta_0\leq
 \inf_{A\in \A_2} \tr(A D^2 u(x) A) &\leq \tr(A_* K A_*)
 =m_*K.
 \end{split}
 \end{equation} 
The local compatibility condition is sharp, since for $u(x)=\tfrac{K}{2}|x|^2$ it is easy to see that $2\sigma_2^{1/2}(D^2u)=m_*K$, and equality holds throughout. 
\end{rem}

The key step in the proof is the following estimate: There exists  $\gamma_0^{\rm loc}>0$ depending only on $n$, $\eta_0$, and $K$ such that
\begin{equation} \label{eq:keyestloc}
\tr(AD^2 u(x) A) \geq \gamma_0^{\rm loc} \tr(A^2), \quad \text{ for all } A\in \A_2\text{ and all } x\in \Omega. 
\end{equation}
 Since degenerate matrices have large trace, \eqref{eq:keyestloc} shows that the corresponding linear operators $\mathcal{L}_A u=\tr(AD^2 u A)$ take large values and do not contribute to the infimum, proving~\eqref{eq:localUE}.

Note that under a strict convexity assumption (i.e., if $D^2u(x)\ge \gamma_0^{\rm loc} I$ for some $\gamma_0^{\rm loc}>0$), estimate \eqref{eq:keyestloc} follows trivially. However, in our case we only know that the eigenvalues of $D^2u(x)$ belong to the admissibility cone $\Gamma_2$, which poses a major difficulty, since $D^2u(x)$ could have negative eigenvalues, making $\tr(AD^2 u(x) A)$ small even if $\tr(A^2)$ is large.
The key point in the sequel is that the possible degeneracies of the matrices in $\mathcal A_2$ are rigid enough that the contribution of these negative directions is compensated by the remaining ones.

\subsection{Strategy} \label{sec:strategy}
To prove estimate \eqref{eq:keyestloc}, let $A\in\mathcal A_2$, and normalize it by its Frobenius norm $\|A\|_F=
\sqrt{\tr(A^2)}$
as in \eqref{def.tilde.A.mu}. 
In particular, $\widetilde A$ has positive eigenvalues bounded~by~1,
and the eigenvalues $\widetilde \mu\in (0,1]^n$ of $\widetilde A^2$ satisfy the normalized equation \eqref{normalized.eq.tilde.mu}.
If the normalized matrix $\widetilde A$ belonged to the class $\A_2$, then \eqref{eq:keyestloc} would readily follow from \eqref{eq:assumloc}. 
However, by Lemma~\ref{lem:trace}, no matrix with unit Frobenius norm belongs to $\A_2$, and we cannot apply \eqref{eq:assumloc} directly to $\widetilde A$.

Instead, we proceed by a perturbation argument.
The idea is to exploit the axial symmetry of the hyperboloid and its asymptotic cone to find a perturbation of $\widetilde A^2$ in the class $\A_2$.
Indeed, given $A\in\A_2$ and its normalization $\widetilde A$, we define $\overline{A^2}$ as the radial projection of $\widetilde{A}^2$ onto the asymptotic sphere~\eqref{norm.section.hyper.cone}. Then, we can set up the perturbation in the 2-dimensional plane spanned by $\overline{A^2}$ and the identity matrix (which spans the axial direction), see Figure \ref{fig.lateral.section}.
 In the next lemma, we present the definition and key properties of $\overline{A^2}$. In particular, we will show that it is nonnegative definite, which justifies the square in the notation $\overline{A^2}$.

 \begin{figure}[t]
 \centering
 \includegraphics[width=0.8\linewidth]{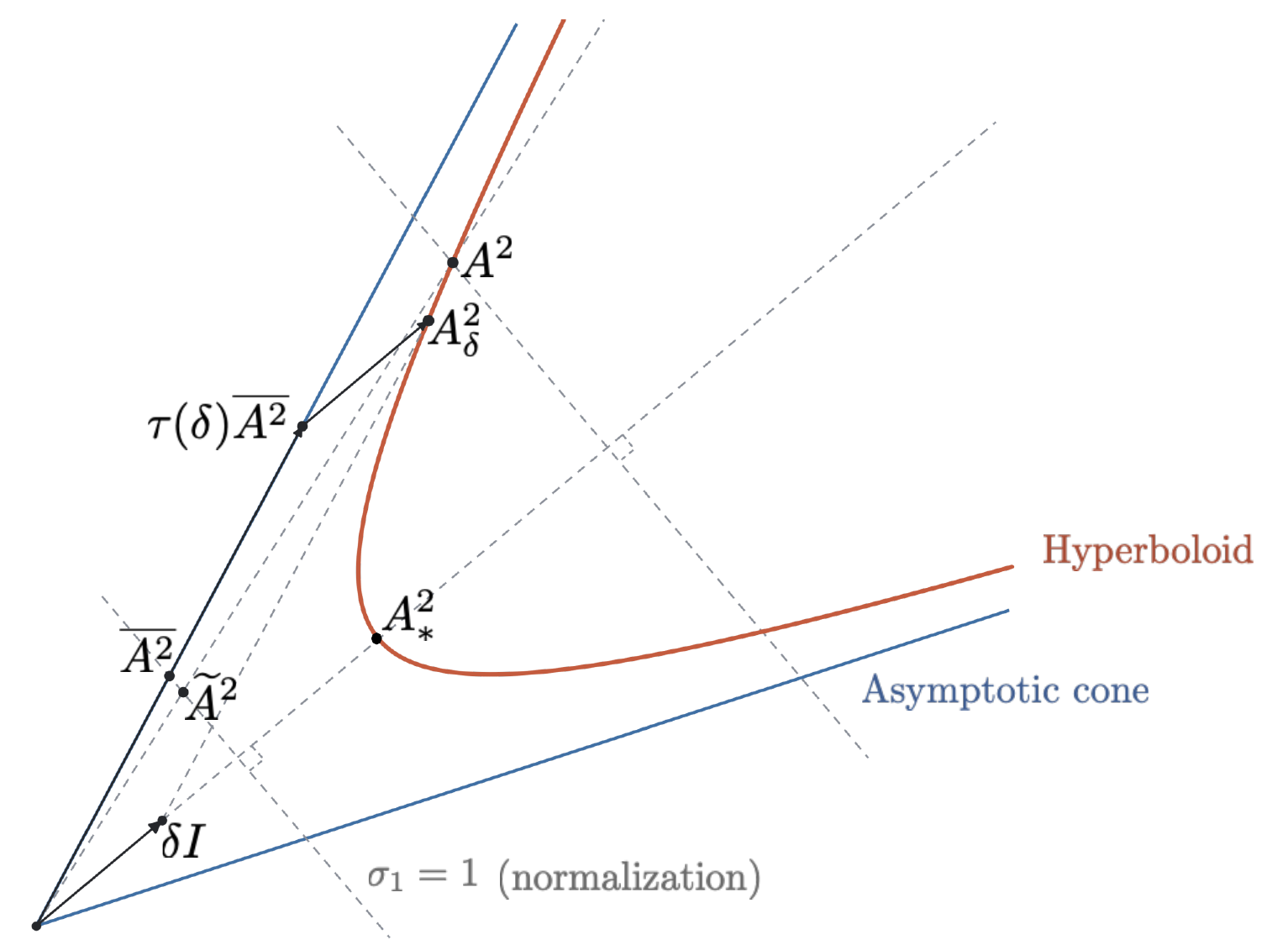}
 \caption{Lateral section showing the construction of the matrix $A_{\delta}$ in \eqref{def.tau.delta}. }
 \label{fig.lateral.section}
 \end{figure}

\begin{lem}\label{lem:perturbation.matrix.and.closeness}
Let $A\in\A_2$ with $\|A\|_F^2=\tr(A^2) > m_*$. Define
\begin{equation}\label{eq:perturbation.and.radius}
\overline{A^2}:=b(A)\,\widetilde A^2+\frac1n \big(1-b(A)\big)\,I,
\qquad\textrm{for}\qquad
b(A):=\bigg(1-\frac{m_*^2}{\|A\|_F^4}\bigg)^{-1/2}>1 .
\end{equation}
Then: 
\begin{enumerate}[(i)]\itemsep3pt
\item The matrix $\overline{A^2}$ is symmetric, commutes with $\widetilde A^2$, and its eigenvalues belong to the asymptotic sphere \eqref{norm.section.hyper.cone}.
In particular $0\leq \overline{A^2}\leq\tfrac2n I$.
\item $\displaystyle
\|\widetilde A^2-\overline{A^2}\|\leq2\sqrt{n(n-1)}\,\|A\|_F^{-4}$.
\item For
$\delta\in(0,t_*)$, we have
\begin{equation}\label{def.tau.delta}
A_{\delta}:=\sqrt{\tau(\delta)\,\overline{A^2}+\delta I}\in\A_2,
\qquad\textrm{for}\qquad
\tau(\delta):=\frac{n-1}{\delta}-\frac{n\delta}{2}>0.
\end{equation}
\end{enumerate}
\end{lem}

\begin{proof}
From its definition, it is clear that $\overline{A^2}$ is symmetric and can be diagonalized simultaneously with $\widetilde A^2$. As a consequence, $\widetilde A^2$ and $\overline{A^2}$ commute. 
Let us denote by $\widetilde\mu$ the eigenvalues of $\widetilde{A}^2$, and by $\overline\mu$ those of $\overline{A^2}$. We can rewrite \eqref{eq:perturbation.and.radius}, as
\begin{equation}\label{eq.mus.lemma}
\overline \mu=\overline c+b\big(\widetilde\mu-\overline c\big)
\end{equation}
for $\overline c$ as in \eqref{def.bar.c}. Let us prove that $\overline\mu$ satisfies \eqref{norm.section.hyper.cone}.
Firstly, taking traces in \eqref{eq:perturbation.and.radius}, and using $\tr\widetilde A^2=1$, we get
$\sigma_1(\overline\mu)=b+(1-b)=1$, the first condition in \eqref{norm.section.hyper.cone}.
Moreover, taking the modulus on both sides of \eqref{eq.mus.lemma}, using Lemma~\ref{lem:sections}, and the fact that
\begin{equation}\label{equation.b.r.tilde.infty}
b(A)=\frac{\overline r_\infty}{\widetilde r(\|A\|_F^2)},
\end{equation}
by the definitions in \eqref{definition-r}, we obtain
$$
|\overline \mu-\overline c| =b\,|\widetilde\mu-\overline c|
=b\,\widetilde r\big(\|A\|_F^2\big)=\overline r_\infty,
$$
the second condition in \eqref{norm.section.hyper.cone}.
Note that  $\overline\mu_i\geq0$ also by Lemma \ref{lem:sections}, completing the proof of~\emph{(i)}.

Let us show \emph{(ii)}. From \eqref{eq:perturbation.and.radius}  and \eqref{equation.b.r.tilde.infty} we have
\[
\|\overline{A^2}-\widetilde A^2\|=(b(A)-1)\big\|\widetilde A^2-\tfrac1nI\big\|
=
\frac{\overline r_\infty-\widetilde r(\|A\|_F^2)}{\widetilde r(\|A\|_F^2)}\,\big\|\widetilde A^2-\tfrac1nI\big\|.
\]
 On the other hand,
\[
\|\widetilde A^2-\tfrac1nI\|=\max_{1\leq i \leq n}|\widetilde\mu_i-\tfrac1n|
\leq|\widetilde\mu-\overline c|=\widetilde r\big(\|A\|_F^2\big).
\]
 Therefore, by \eqref{eq:radialgap},
$$
\|\widetilde A^2-\overline{A^2}\|
\leq\overline r_\infty-\widetilde r\big(\|A\|_F^2\big)
\leq\frac{2\sqrt{n(n-1)}}{\|A\|_F^4},
$$
as desired.

For the proof of \emph{(iii)}, observe that the eigenvalues of $A_{\delta}^2=\tau \overline{A^2}+\delta I$ are
$\mu_i^\delta=\tau\overline\mu_i+\delta\geq\delta>0$, since $\overline\mu_i\geq0$ by \emph{(i)}.
Hence $A_{\delta}$ is well defined and $A_{\delta}\in\S^n_+$. Let us verify that this $\mu^\delta$ satisfies~\eqref{eq:classA2}. Indeed,
by \eqref{alternative.asymp.esphere.def}, we have $\sigma_1(\overline\mu)=1$ and $|\overline\mu|^2=\tfrac{1}{n-1}$, which gives
$$
\sigma_1(\mu^\delta)=\tau+n\delta,
\qquad
|\mu^\delta|^2=\tau^2|\overline\mu|^2+2\tau\delta\,\sigma_1(\overline\mu)+n\delta^2
=\frac{\tau^2}{n-1}+2\tau\delta+n\delta^2.
$$
Then we can compute
$$
\frac{1}{n-1}\,\sigma_1(\mu^\delta)^2-|\mu^\delta|^2
=\frac{(\tau+n\delta)^2-\tau^2}{n-1}-2\tau\delta-n\delta^2
=\frac{2\tau\delta+n\delta^2}{n-1}=2,
$$
by the choice of $\tau(\delta)$. By Lemma~\ref{lem:char}, $A_{\delta}\in\A_2$ and the proof of the lemma is complete.
\end{proof}

For convenience in the sequel, we write the perturbation \eqref{def.tau.delta} in terms of $\widetilde{A}^2$ instead~of~$\overline{A^2}$.

\begin{cor}\label{cor:pert.tildeA}
Let $A\in\A_2$ and $\delta\in(0,\frac{t_*}{2}]$, and assume
\begin{equation}\label{eq:threshold.T0}
\|A\|_F^{2} \geq \frac{2(n-1)}{\delta}.
\end{equation}
Then, the matrix $A_{\delta}\in\A_2$ from \eqref{def.tau.delta} satisfies
\begin{equation}\label{def.B.delta.tildeA}
A_{\delta}=\sqrt{\tau'\,\widetilde{A}^2+\delta'\, I},
\qquad
\tau':=\tau(\delta)\,b(A),
\qquad
\delta':=\delta-\frac{\tau(\delta)\big(b(A)-1\big)}{n},
\end{equation}
with
\begin{equation}\label{eq:primed.brackets}
\frac{n-1}{2\,\delta}\leq\tau'\leq\frac{2\,(n-1)}{\sqrt{3}\,\delta}
\qquad\textrm{and}\qquad
\frac{\delta}{2}\leq\delta'\leq\delta .
\end{equation}
\end{cor}

\begin{proof}
First, note that \eqref{eq:threshold.T0} implies $\|A\|_F^2\geq\tfrac{2(n-1)}{\delta}\geq 2m_*$, so Lemma~\ref{lem:perturbation.matrix.and.closeness} applies.
The identity \eqref{def.B.delta.tildeA} follows by substituting \eqref{eq:perturbation.and.radius} into \eqref{def.tau.delta}.
Let us prove \eqref{eq:primed.brackets}. Since $\delta^2\leq\tfrac{n-1}{2n}$, we have $\tfrac{n\delta}{2}\leq\tfrac{n-1}{2\delta}$, and hence $\tau(\delta)\geq\tfrac{n-1}{2\delta}$.
Together with $b(A)>1$, this gives the lower bound for $\tau'$. By \eqref{eq:threshold.T0} and $\delta^2\leq\tfrac{n-1}{2n}$,
\[
\frac{2n(n-1)}{\|A\|_F^4}\leq\frac{n\,\delta^{2}}{2(n-1)}\leq\frac14 .
\]
In particular, $b(A)<2/\sqrt{3}$, and the upper bound for $\tau'$
follows from $$\tau(\delta)\leq\frac{n-1}{\delta}.$$ To conclude, let us show that $\delta'\in[\delta/2,\delta]$. Using again that $b(A)<2/\sqrt{3}<(1+\sqrt{5})/2$, we have $b(A)^2-b(A)-1<0$ and
\[
b(A)-1\leq1-\frac{1}{b(A)^2}=\frac{2n(n-1)}{\|A\|_F^4}=\frac{m_*^2}{\|A\|_F^4}.
\]
Therefore,
\[
0\leq\frac{\tau(\delta)\big(b(A)-1\big)}{n}
\leq\frac{2(n-1)^2}{\delta\|A\|_F^{4}}
\leq\frac{\delta}{2},
\]
which completes the proof.
\end{proof}

\begin{rem}
\label{rem:path}
For fixed $A\in\A_2$, the matrices $A_{\delta}$ describe a smooth arc on the hyperboloid~\eqref{eq:classA2} contained in the axial plane through $A$, and joining $A$ to the vertex $A_{*}$ defined in \eqref{def.A.star}, see Figure~\ref{fig.lateral.section}. Indeed, $\delta'$ vanishes precisely at
$$
\delta_0(A)^2=\frac{2(n-1)}{n}\cdot\frac{b(A)-1}{b(A)+1},
$$
where $A_{\delta_0}=A$, while $A_{\delta}\to A_{*}$ as $\delta\to t_*$, which gives $\tau(\delta)\to0$. Along this arc,
$$
\|A_{\delta}\|_F^2=\tr\big(A_{\delta}^2\big)=\tau(\delta)+n\delta=\frac{n-1}{\delta}+\frac{n\delta}{2}
$$
is strictly decreasing, from $\|A\|_F^2$ to the minimum value
$m_*$ given in Lemma~\ref{lem:trace}, attained at the vertex.
\end{rem}

We are now ready to prove estimate \eqref{eq:keyestloc}. Fix
\begin{equation}\label{fixing.delta.local.new}
\delta:=\frac{\eta_0}{2nK},
\end{equation}
and observe that $\delta\leq\tfrac{t_*}{2}$ by the compatibility condition \eqref{calculation.compatibility.local}. Hence, $\delta$ is under the hypotheses of Corollary~\ref{cor:pert.tildeA}.
Let $T_0:=T_0(n,\delta)=\frac{2(n-1)}{\delta}$ be as in \eqref{eq:threshold.T0}. If $\|A\|_F^2< T_0$, then \eqref{eq:assumloc} directly gives
\begin{equation}\label{local.case.small.A}
\tr\big(AD^2u(x)A\big)\geq\eta_0
\geq\frac{\eta_0}{T_0}\,\tr(A^2)
=\frac{\eta_0\,\delta}{2(n-1)}\,\tr(A^2).
\end{equation}
Assume henceforth $\|A\|_F^2\geq T_0$, and let $A_{\delta}\in\A_2$, $\tau'$, $\delta'$ be as in \eqref{def.B.delta.tildeA}.
Applying \eqref{eq:assumloc} to $A_{\delta}\in\A_2$, and using that $\Delta u(x)\leq nK$ by the semiconcavity of $u$ along with $\delta'\leq\delta$ and \eqref{fixing.delta.local.new}, we have
\begin{equation}\label{comentario.sobre.semiconcavidad.Laplacian}
\begin{split}
\eta_0\leq\tr\big(A_{\delta}D^2u(x)A_{\delta}\big)
&=\tr\big(A_{\delta}^2D^2u(x)\big)
=\tau'\tr\big(\widetilde{A}^2D^2u(x)\big)+\delta'\Delta u(x)\\
&\leq\tau'\tr\big(\widetilde{A}D^2u(x)\widetilde{A}\big)+n\delta K
\leq\tau'\,\frac{\tr(AD^2u(x)A)}{\|A\|_F^2}+\frac{\eta_0}{2},
\end{split} 
\end{equation}
which, by \eqref{eq:primed.brackets}, gives
\begin{equation}\label{local.case.big.A}
\tr\big(A\,D^2u(x)\,A\big)
\geq\frac{\eta_0}{2\tau'}\,\tr(A^2)
\geq\frac{\eta_0\sqrt{3}\,\delta}{4(n-1)}\,\tr(A^2).
\end{equation}
Combining \eqref{local.case.small.A} and \eqref{local.case.big.A}, by the definition of $\delta$ in \eqref{fixing.delta.local.new} we can take 
\begin{equation} \label{eq:gamma-classical}
\gamma_0^{\rm loc}(n,K,\eta_0):= \min\bigg\{\frac{1}{2(n-1)},\frac{\sqrt{3}}{4(n-1)}\bigg\}\eta_0\delta=\frac{\sqrt{3}\eta_0^2}{8n(n-1)K}, 
\end{equation}
and \eqref{eq:keyestloc} follows.

\begin{rem}
The proof of \emph{(iii)} in Lemma \ref{lem:perturbation.matrix.and.closeness} only uses the two values
$\sigma_1(\overline\mu)=1$ and $|\overline\mu|^2=\tfrac{1}{n-1}$ to determine $A_{\delta}\in\A_2$. Since these are constant on the asymptotic sphere by \eqref{alternative.asymp.esphere.def}, the function $\tau(\delta)$ depends only on $\delta$ and not on the matrix $\overline{A^2}$. In particular, this allows us to fix $\delta$ as in \eqref{fixing.delta.local.new}, depending only on $n,K,\eta_0$ and independently of $A$, which, in turn, yields a constant  $\gamma_0>0$ that does not degenerate as $\|A\|_F\to\infty$.
Ultimately, this is possible because the projection $\overline{A^2}$ lies in the normalized section of the asymptotic cone, whose sections are homothetic 
and all scale to the asymptotic sphere.
\end{rem}

\subsection{From local to nonlocal} 

The previous argument uses properties satisfied by the trace of a matrix that are not available in the nonlocal setting; mainly, the use of the linearity of the trace in the coefficient matrix along with the semiconcavity of $u$ to control the axial perturbation in \eqref{comentario.sobre.semiconcavidad.Laplacian}.
However, the general strategy can be adapted. To motivate our approach, we recall that the trace operator can be written as an integral average of second-order directional derivatives on the sphere. 

\begin{lem} \label{lem:average}
 For any $B\in \S^n$,
 $$
 \frac{1}{n} \tr(B) = \fint_{\mathbb{S}^{n-1}} \langle B e, e\rangle \, d\H_e^{n-1}.
 $$
 \end{lem}

 \begin{proof}
 Let $B \in \S^n$. We have
 \begin{align*}
 \int_{\mathbb{S}^{n-1}} \langle B e, e\rangle\, d\H_e^{n-1} 
 &=\sum_{i,j=1}^n b_{ij} \int_{\mathbb{S}^{n-1}} e_i e_j \, d\H^{n-1}_e.
 \end{align*}
 By the symmetry of the sphere,
 $$
\int_{\mathbb{S}^{n-1}} e_i e_j \, d\H^{n-1}_e = 0, \quad \text{ for all } i\neq j,
 $$
and
 $$
 n\int_{\mathbb{S}^{n-1}} |e_i|^2 \, d\H^{n-1}_e = \sum_{j=1}^n \int_{\mathbb{S}^{n-1}} |e_j|^2 \, d\H^{n-1}_e = \omega_{n-1},
 $$
 which gives the result.
 \end{proof}

Using Lemma \ref{lem:average}, and the notation
 $\Theta u(x,e)$ and $\I u(x,y)$ for the second-order directional derivatives on the sphere and Hessian quadratic form, respectively, we obtain a local counterpart of the representation formula in Lemma~\ref{lem:newrep} as follows:
\begin{equation}\label{representation.local}\begin{split}
\frac{1}{n} \tr(A^t D^2 u(x) A) &= \fint_{\mathbb{S}^{n-1}} 
\big\langle D^2 u(x) (A\theta), (A\theta)\big\rangle \, d\H^{n-1}_\theta\\
 &= \fint_{\mathbb{S}^{n-1}} 
 \I u(x,A\theta)
 \, d\H^{n-1}_\theta\\
 &= \fint_{\mathbb{S}^{n-1}} |A \theta|^2 \Theta u\Big(x,\frac{A\theta}{|A\theta|} \Big) \,d\H^{n-1}_\theta.
 \end{split}
\end{equation} 
This observation is the starting point of our nonlocal analysis. Our proof of the nonlocal key estimate in Theorem \ref{thm:keyest} uses two main ingredients to adapt \eqref{comentario.sobre.semiconcavidad.Laplacian}. The first one is the nonlocal representation formula in terms of one-dimensional fractional Laplacians, Lemma \ref{lem:newrep}, and the second one, an operator comparison result, Proposition \ref{prop:tech-lem}, that replaces the trace linearity in the coefficient matrix, which will allow us to control the axial remainders in the nonlocal case.

\section{Gaussian kernels and operator comparison} \label{app:tech-lem}

We will prove the following operator comparison result, which will play a major role in Section~\ref{sec:keyest}.

\begin{prop} \label{prop:tech-lem}
 Let $s\in (0,1)$ and assume $u\in C(\Rn)\cap L_{\omega_s}^1(\Rn)$ satisfies \eqref{NSC} at a point $x$ where there is a $C^2$-function touching $u$ from above.
 Let
 $Q_1,Q_2\in\mathcal S^n_+$ be such that $Q_2^2> Q_1^2$, and assume that
 \begin{equation}\label{eq:finiteness-hyp}
 \mathcal{L}^s_{Q_i}u(x)>-\infty, \qquad i=1,2.
 \end{equation}
 Then, 
\begin{equation}\label{comparison.estimate}
 \mathcal{L}_{Q_2}^su(x)- \mathcal{L}_{Q_1}^su(x)
\leq 
c_{K,s}
\left(\tr(Q_2^2-Q_1^2)\right)^s,
\end{equation}
where 
\[
c_{K,s}:=K\cdot \frac{\omega_{n-1}}{4 n^s}\cdot\frac{C_{n,s}}{1-s},
\]
and $c_{K,s}\to K$ as $s\to 1$.
\end{prop}

Observe that the operator comparison estimate \eqref{comparison.estimate} is one-sided, which is inherited from the one-sided nonlocal semiconcavity hypothesis \eqref{NSC}. 
Obtaining a two-sided estimate would require additional information on $u$, such as a lower bound in \eqref{NSC}, which could be seen as a form of nonlocal convexity or semiconvexity. 
However, such conditions are not natural in our context, as they avoid the degeneracy cases in the class $\A_2$. In any case, the proof of the key estimate in Section \ref{sec:keyest} only requires the upper bound in \eqref{comparison.estimate}.

\medskip

The proof of Proposition~\ref{prop:tech-lem} uses the representation formula from Lemma~\ref{lem:newrep}, and it reduces to controlling the difference
\[
 \mathcal{L}_{Q_2}^su(x)- \mathcal{L}_{Q_1}^su(x)
= \frac{C_{n,s}}{4(1-s)}\int_{\mathbb S^{n-1}} \bigl(\I^s u (x,Q_2\theta)-\I^s u (x,Q_1\theta)\bigr) \,d \H^{n-1}_\theta,
\]
where $\I^s u$ is defined in \eqref{def.alternate.I.s}.
Recall that $\I^su(x,\cdot)$ is homogeneous of degree $2s$, see \eqref{I.s-homogeneous}.
Next, we recall an elementary formula reducing integrals of homogeneous functions against radial weights to integrals over the sphere.
This is a classical idea in the analysis of homogeneous distributions (see, for instance, \cite{Folland} for a historical note).

\begin{lem}\label{lem:hom-radial}
Let $\J:\mathbb R^n\setminus\{0\}\to\mathbb R$ be homogeneous of degree $\alpha\in\mathbb{R}$, i.e., such that 
$\J(rz)=r^{\alpha}\J(z)$ for all $r>0$ and $z\in\mathbb R^n\setminus\{0\}$, with $\J\in L^1(\mathbb S^{n-1})$. Let $w:(0,\infty)\to[0,\infty)$ be measurable with
\[
\kappa_w:=\int_0^\infty w(r)\,r^{n+\alpha-1}\,dr<\infty .
\]
Then $\J w(|\cdot|)\in L^1(\mathbb R^n)$ and
\[
\int_{\mathbb R^n}\J(z)\,w(|z|)\,dz=\kappa_w\int_{\mathbb S^{n-1}}\J(\theta)\,d\mathcal H^{n-1}_\theta.
\]
\end{lem}
\begin{proof}
By polar coordinates and $\alpha$-homogeneity, 
\[
\int_{\mathbb R^n}\J(z)\,w(|z|)\,dz = \int_0^\infty \left(\int_{\mathbb S^{n-1}}\J(\theta)\,d\H^{n-1}_\theta \right)\, w(r)\,r^{n+\alpha-1}\,dr
=\kappa_w\int_{\mathbb S^{n-1}}\J(\theta)\,d\mathcal H^{n-1}_\theta,
\]
as desired.
\end{proof}

We will use this lemma with Gaussian kernels. 
For $Q\in \S^n_+$, we define the Gaussian kernel
$$
G_Q(y):=
\frac{1}{(2\pi)^{n/2}(\det Q)^{1/2}} e^{-\frac12\langle Q^{-1}y,y\rangle}, \qquad \text{ for } y\in \Rn.
$$
Notice that $G_Q$ is even, and
$$
\int_{\mathbb R^n}G_Q(y)\,dy=1.
$$
For a weight given by a Gaussian, Lemma \ref{lem:hom-radial} particularizes to the following.

\begin{cor} \label{lem:gauss-sphere}
Let $s\in(0,1)$, and $\mathcal{J}:\mathbb R^n\setminus\{0\}\to\mathbb R$ be homogeneous of degree $2s$. Assume that $\J \in L^1(\mathbb S^{n-1})$.
Then, for all $Q\in \S^n_+$,
$$
\int_{\mathbb R^n} \J(y)\,G_{Q^2}(y)\,dy = \kappa_{n,s} \int_{\mathbb S^{n-1}} \J(Q\theta)\, d\mathcal H^{n-1}_\theta,
$$
where
\begin{equation}\label{definition-kappa}
\kappa_{n,s}:=\frac{2^{s-1}}{\pi^{n/2}}\,\Gamma\Big(\frac n2+s\Big).
\end{equation}
\end{cor}

\begin{proof}
Since $Q$ is positive definite, $|Q\theta|$ is bounded above and below by positive constants on $\mathbb S^{n-1}$, and the map
$
\theta\mapsto \frac{Q\theta}{|Q\theta|}
$
is a smooth diffeomorphism of $\mathbb S^{n-1}$. This, together with the $2s$-homogeneity of $\J$, and $\J\in L^1(\mathbb S^{n-1})$, shows that $\J_Q(y)=\J(Qy)$ is under the hypotheses of Lemma \ref{lem:hom-radial}.
Applying the lemma to $\J_Q$ with a Gaussian weight $w(r)=(2\pi)^{-n/2}e^{-r^2/2}$, and taking into account that $Q$ is symmetric, yields
\[
\int_{\mathbb R^n}\J(y)G_{Q^2}(y)\,dy =\int_{\mathbb R^n}\J(Qz)\,w(|z|)\,dz
=\kappa_{n,s} \int_{\mathbb S^{n-1}}\J(Q\theta)\,d\mathcal H^{n-1}_\theta,
\]
where
\begin{equation*}
\kappa_{n,s}=\frac{1}{(2\pi)^{n/2}}\int_0^\infty r^{n-1+2s}e^{-r^2/2}\,dr
=\frac{2^{s-1}}{\pi^{n/2}}\,\Gamma\Big(\frac n2+s\Big).
\end{equation*}
\end{proof}

 Combining Lemma~\ref{lem:newrep} with Corollary~\ref{lem:gauss-sphere}, applied to the $2s$-homogeneous function $\I^su(x,\cdot)$, we can represent the operator $\K^su(x)$ as the infimum of the Gaussian averages of the $\I^su(x,\cdot)$ over the Gaussian kernels $G_{A^2}$, with $A\in\A_2$.

\begin{cor}
Let $A\in\A_2$.
We have the representations
\begin{equation}\label{representation-L}
\begin{split}
\mathcal L^s_Au(x)&=
\frac{C_{n,s}}{4(1-s)\,\kappa_{n,s}}
\int_{\Rn}\I^su(x,y)\,G_{A^2}(y)\,dy
\\
&=
\frac{C_{n,s}}{4(2\pi)^\frac{n}{2}(1-s)\,\kappa_{n,s}}
\int_{\Rn}\I^su(x,Ay)\,e^{-\frac{|y|^2}{2}}dy.
\end{split}
\end{equation}
 Consequently,
\begin{equation}\label{eq:K.gaussian}
\K^su(x)=
\frac{C_{n,s}}{4(2\pi)^\frac{n}{2}(1-s)\,\kappa_{n,s}}
\inf_{A\in\A_2}\int_{\Rn}\I^su(x,Ay)\,e^{-\frac{|y|^2}{2}}dy.
\end{equation}
\end{cor}

We also recall the following standard convolution property for the kernels $G_Q$.

\begin{lem} \label{lem:gauss-conv}
Let $Q_1,Q_2\in\mathcal S^n_+$. Then, $G_{Q_1}*G_{Q_2}=G_{Q_1+Q_2}.$
\end{lem}

\begin{proof}
Given $y$, let us take $z_0=z_0(y)$ such that $(Q_1^{-1}+Q_2^{-1})z_0=Q_1^{-1}y,$ and observe that
\[
\big\langle Q_1^{-1}(y-z),(y-z)\big\rangle +\big\langle Q_2^{-1}z,z\big\rangle
=\big\langle (Q_1^{-1}+Q_2^{-1})(z-z_0),(z-z_0)\big\rangle +\big\langle (Q_1+Q_2)^{-1}y,y\big\rangle.
\]
By definition of convolution and the above identity, 
\[
\begin{split}
(G_{Q_1} *\, & G_{Q_2})(y)=\\
&=
\frac{1}{(2\pi)^n(\det Q_1\det Q_2)^{\frac12}}
\int_{\mathbb R^n} \exp\Big[{-\tfrac12 \big(\langle Q_1^{-1}(y-z),(y-z)\rangle +\langle Q_2^{-1}z,z\rangle \big)}\Big]
dz\\
&=\frac{1}{(2\pi)^n(\det Q_1\det Q_2)^{\frac12}}\,
\exp\Big[-\tfrac12\big\langle (Q_1+Q_2)^{-1}y,y\big\rangle\Big]\\
&\hspace{119pt}\int_{\mathbb R^n}
\exp\Big[-\tfrac12\big\langle (Q_1^{-1}+Q_2^{-1})(z-z_0),(z-z_0)\big\rangle\Big]dz.
\end{split}
\]
For any matrix $P\in\S_{+}^n$, we can perform the change of variables $x=P^{1/2}z$, and use the one-dimensional Gaussian integral to obtain
\[
\int_{\mathbb R^n}
e^{-\frac12\langle Pz,z\rangle}\,dz
=\frac{1}{(\det P)^{1/2}}\int_{\mathbb R^n}
e^{-\frac12|x|^2}\,dx
 =\frac{1}{(\det P)^{1/2}}\prod_{j=1}^n\int_{\mathbb R} e^{-\frac{1}{2}x_j^2}\,dx_j
=\frac{(2\pi)^{n/2}}{(\det P)^{1/2}}.
\]
Applying this identity with $P=Q_1^{-1}+Q_2^{-1}$, and taking determinants in the matrix identity
$$Q_1^{-1}+Q_2^{-1}=Q_1^{-1}(Q_1+Q_2)Q_2^{-1},$$ we obtain $(G_{Q_1}*G_{Q_2})(y)=G_{Q_1+Q_2}(y)$ as desired.
\end{proof}

We are ready to prove our operator comparison result.

\begin{proof}[Proof of Proposition \ref{prop:tech-lem}]
Assume that a $C^2$-function touches $u$ from above at $x\in \Omega$. 
By the representation formula (Lemma~\ref{lem:newrep}),
\begin{equation} \label{eq:differ-ops}
 \mathcal{L}_{Q_2}^su(x)- \mathcal{L}_{Q_1}^su(x)
= \frac{C_{n,s}}{4(1-s)}\int_{\mathbb S^{n-1}} \bigl(\I^s u (x,Q_2\theta)-\I^s u (x,Q_1\theta)\bigr) \,d \H^{n-1}_\theta
\end{equation}
in the classical sense. Hence, we need to control the difference in the integrand.

Now, since $Q_2^2- Q_1^2>0$ by hypothesis, Lemma~\ref{lem:gauss-conv}, yields
$$
G_{Q_2^2}=G_{Q_1^2}*G_{Q_2^2-Q_1^2}.
$$
Therefore,
$$
\int_{\Rn}\I^s u(x,y)G_{Q_2^2}(y)\,dy = \int_{\Rn}\int_{\Rn}\I^s u (x,y+z)G_{Q_1^2}(y)G_{Q_2^2-Q_1^2}(z)\,dydz.
$$
Since $G_{Q_2^2-Q_1^2}(z)$ has integral one, and it is even, we have
\begin{align*}
&\int_{\Rn} \I^s u (x,y)G_{Q_2^2}(y)\,dy -\int_{\Rn}\I^s u (x,y)G_{Q_1^2}(y)\,dy \\
&\qquad=
\frac12 \int_{\Rn}\int_{\Rn} \big(\I^s u (x,y+z)+\I^s u (x,y-z)-2 \I^s u (x,y)\big) \,G_{Q_1^2}(y)\,G_{Q_2^2-Q_1^2}(z)\,dydz.
\end{align*}
Using \eqref{NSC}, we obtain
\begin{equation}\label{lem:comp.eq1}
\begin{split}
&\int_{\Rn} \I^s u (x,y)G_{Q_2^2}(y)\,dy -\int_{\Rn}\I^s u (x,y)G_{Q_1^2}(y)\,dy \\
&\qquad\qquad\leq
K\int_{\Rn} G_{Q_1^2}(y)\,dy \int_{\Rn} |z|^{2s}G_{Q_2^2-Q_1^2}(z)\,dz
=K\int_{\Rn} |z|^{2s}G_{Q_2^2-Q_1^2}(z)\,dz.
\end{split}
\end{equation}

By Remark \ref{remark.I.s.in.L1}, the $2s$-homogeneous function $\I^s u(x,\cdot)$ is in $L^1(\mathbb S^{n-1})$, and satisfies the 
assumptions of Corollary~\ref{lem:gauss-sphere}. Hence, 
\begin{equation}\label{lem:comp.eq2}
\int_{\mathbb R^n} \I^s u (x,y)G_{Q_i^2}(y)\,dy = \kappa_{n,s} \int_{\mathbb S^{n-1}} \I^s u (x,Q_i\theta)\, d\mathcal H^{n-1}_\theta,\qquad\textrm{for}\ i=1,2.
\end{equation}
Moreover, for any $Q\in\S^n_+$, again using Corollary~\ref{lem:gauss-sphere}, along with Jensen's inequality, and Lemma~\ref{lem:average}, 
\begin{equation}\label{lem:comp.eq3}
 \begin{split}
\int_{\Rn} |z|^{2s}G_{Q^2}(z)\,dz
&=\kappa_{n,s} \int_{\mathbb S^{n-1}} |Q\theta|^{2s}\, d\mathcal H^{n-1}_\theta
=\kappa_{n,s}\,\omega_{n-1} \fint_{\mathbb S^{n-1}} |Q\theta|^{2s}\, d\mathcal H^{n-1}_\theta
\\
&\leq\kappa_{n,s}\,\omega_{n-1} \left(\fint_{\mathbb S^{n-1}} |Q\theta|^{2}\, d\mathcal H^{n-1}_\theta\right)^s
=\kappa_{n,s}\,\omega_{n-1}\,\left(\frac{\tr(Q^2)}{n}\right)^s.
\end{split} 
\end{equation}

Together, \eqref{lem:comp.eq1}, \eqref{lem:comp.eq2}, and \eqref{lem:comp.eq3} with $Q=\big(Q_2^2-Q_1^2\big)^{1/2}$ 
(which is well defined since $Q_2^2>Q_1^2$) yield
$$
\int_{\mathbb S^{n-1}} \bigl(\I^s u (x,Q_2\theta)-\I^s u (x,Q_1\theta)\bigr) \,d \H^{n-1}_\theta 
\leq  K\, \frac{\omega_{n-1}}{n^s}\,\Big(\tr(Q_2^2-Q_1^2)\Big)^s,
$$
which, combined with \eqref{eq:differ-ops}, gives
$$
 \mathcal{L}_{Q_2}^su(x)- \mathcal{L}_{Q_1}^su(x)
\leq K\,\frac{\omega_{n-1}}{4\,n^s}\,\frac{C_{n,s}}{(1-s)}
\left(\tr(Q_2^2-Q_1^2)\right)^s
=c_{K,s}\left(\tr(Q_2^2-Q_1^2)\right)^s.
$$
By \eqref{eq:Cns.limit}, $c_{K,s}\to K$ as $s\to 1$.
\end{proof}

\section{Proof of the key estimate and main result} \label{sec:keyest}

In this section, we adapt to the nonlocal setting the strategy discussed in Section~\ref{sec:idealocal}. 
The main result, Theorem~\ref{thm:main}, follows from the following key estimate.

\begin{thm} \label{thm:keyest}
Let $u$ be as in Theorem~\ref{thm:main}. Then, for all $A\in \A_2$,
\begin{equation*} 
\mathcal{L}_A^s u \geq \gamma_0 \|A\|_F^{2s} \quad \text{ in } \Omega,
\end{equation*}
in the viscosity sense, where $\gamma_0>0$ depends only on $n$, $s$, $K$, and $\eta_0$.
Moreover, $\gamma_0\to \gamma_0^{\rm loc}$ as $s\to1$, where $\gamma_0^{\rm loc}$ is given in \eqref{eq:gamma-classical}.
\end{thm}

\begin{proof}
Assume that a $C^2$-function touches $u$ from above at $x\in\Omega$. By Lemma~\ref{lem:classicalsense},
$\mathcal K^su(x)\geq\eta_0$ holds in the classical sense. In particular,
$$
\mathcal L_A^su(x)\geq\eta_0, \quad \text{ for all } A\in \A_2.
$$

Given $A\in\A_2$, let
$$
\widetilde A:=\frac{1}{\|A\|_F} A.
$$
Let $c_{K,s}$ be the constant in Proposition~\ref{prop:tech-lem}, and fix
\begin{equation}
\delta:=\left(\frac{\eta_0}{2c_{K,s}n^s}\right)^{1/s}.
\label{eq:fixing-delta-gaussian}
\end{equation}
The compatibility condition in \eqref{compatibility-non-local} can be written as 
$$
\eta_0 \leq c_{K,s}\,m_*^s.
$$
Hence, $\delta \leq t_*/2,$ so it satisfies the assumption of Corollary~\ref{cor:pert.tildeA}.

Assume first that
$$
\|A\|_F^2<\frac{2(n-1)}{\delta}.
$$
By the homogeneity of the operator,
\begin{equation}\label{main.est.first.half}
\mathcal L_{\widetilde A}^su(x) = \frac{1}{\|A\|_F^{2s}}\mathcal L_A^su(x) 
\geq \left(\frac{1}{2(n-1)}\right)^s \eta_0 \delta^s
\geq \frac{1}{2} \left(\frac{\sqrt3}{2(n-1)}\right)^s \eta_0\delta^s.
\end{equation}
Assume now that
$$
\|A\|_F^2\geq\frac{2(n-1)}{\delta}.
$$
Let $A_{\delta}\in\A_2$, $\tau'$, and $\delta'$ be as in Corollary~\ref{cor:pert.tildeA}, so that
$$
A_{\delta}^2=\tau'\widetilde A^2+\delta'I,
\qquad
\frac{n-1}{2\delta}\leq\tau'
\leq\frac{2(n-1)}{\sqrt3\,\delta},
\qquad
\frac{\delta}{2}\leq\delta'\leq\delta.
$$
Set
$$
B:=\frac{1}{\sqrt{\tau'}}A_{\delta}.
$$
Then
$$
B^2=\widetilde A^2+\frac{\delta'}{\tau'}I,
$$
and, since $A_{\delta}\in\A_2$,
\begin{equation}\label{main.est.first.half.for.B}
\mathcal L_B^su(x)=\frac{1}{(\tau')^s}\mathcal L_{A_{\delta}}^su(x)
\geq \frac{\eta_0}{(\tau')^s}.
\end{equation}

Note that $\widetilde A$ and $B$ satisfy hypothesis \eqref{eq:finiteness-hyp}: for $\widetilde A$ by homogeneity, i.e.,
$$
\mathcal L^s_{\widetilde A}u(x)=\|A\|_F^{-2s}\mathcal L^s_{A}u(x)
\geq\|A\|_F^{-2s}\eta_0,
$$
and for $B$ by \eqref{main.est.first.half.for.B}. Moreover, $B^2-\widetilde A^2=\frac{\delta'}{\tau'}\,I>0$.
Applying Proposition~\ref{prop:tech-lem} with $Q_1=\widetilde A$ and $Q_2=B$, we obtain
\begin{align*}
\mathcal L_{\widetilde A}^su(x) &\geq \mathcal L_B^su(x) - c_{K,s}\left(\frac{n\delta'}{\tau'}\right)^s
 \geq \frac{\eta_0-c_{K,s}n^s(\delta')^s}{(\tau')^s}
\geq \frac{\eta_0-c_{K,s}n^s\delta^s}{(\tau')^s} = \frac{\eta_0}{2(\tau')^s},
\end{align*}
where in the last equality we used \eqref{eq:fixing-delta-gaussian}. Using the upper bound for $\tau'$,
$$
\mathcal L_{\widetilde A}^su(x) \geq \frac{1}{2} \left(\frac{\sqrt3}{2(n-1)}\right)^s \eta_0\delta^s,
$$
the same estimate we obtained in \eqref{main.est.first.half}.

Defining
$$
\gamma_0:=\frac{1}{2} \left(\frac{\sqrt3}{2(n-1)}\right)^s \eta_0\delta^s = \left(\frac{\sqrt3}{2(n-1)}\right)^s \frac{\eta_0^2}{4c_{K,s}n^s}, 
$$
and since
$\mathcal L_A^su=\|A\|_F^{2s}\mathcal L_{\widetilde A}^su$, we conclude that
$$
\mathcal L_A^su(x)\geq\gamma_0\|A\|_F^{2s}.
$$

It remains to study the asymptotic behavior of $\gamma_0$ as $s\to 1$. Since $c_{K,s}\to K$,
$$
\gamma_0 \to \frac{\sqrt3\,\eta_0^2}{8n(n-1)K} =\gamma_0^{\rm loc},
$$
where the local constant $\gamma_0^{\rm loc}$ is defined in \eqref{eq:gamma-classical}.
\end{proof}

\medskip

Finally, we prove the main theorem, using the key estimate.

\begin{proof}[Proof of Theorem~\ref{thm:main}]
Assume that a $C^2$-function touches $u$ from above at $x\in \Omega$. 
By Theorem~\ref{thm:keyest}, for any $A\in \A_2$, we have
\begin{equation} \label{eq:keyest2}
\mathcal{L}_A^s u (x) = \frac{C_{n,s}}{2} \int_{\R^n} \frac{\delta u(x,Ay)}{|y|^{n+2s}}\, dy\geq \gamma_0 \|A\|_F^{2s},
\end{equation}
 where $\gamma_0>0$ depends only on $n$, $s$, $K$, and $\eta_0$. Fix $\Lambda\geq 1$ to be chosen. Note that if $A\in \A_2\setminus \mathcal{N}(\Lambda)$, then there is $\lambda_i:=\lambda_i(A)$ such that $\lambda_i<\Lambda^{-1}$ or $\lambda_i>\Lambda$. If $\lambda_i<\Lambda^{-1}$, then by Lemma~\ref{lema1}, we have $\lambda_j\geq \lambda_i^{-1}>\Lambda$, for all $j\neq i$. Hence, in both cases, it holds that
 $$
\|A\|_F^2=\tr(A^2)\geq \Lambda^2,
 $$
 which together with \eqref{eq:keyest2} yields
 \begin{equation} \label{eq:oplarge}
 \frac{C_{n,s}}{2} \inf_{A\in \A_2\setminus \mathcal{N}(\Lambda)} \int_{\R^n} \frac{\delta u(x,Ay)}{|y|^{n+2s}}\, dy \geq \gamma_0\Lambda^{2s}.
 \end{equation}
 
Now we use \eqref{compatibility-non-local}.
Taking $\Lambda\geq 1$ large enough so that
\begin{equation}\label{choose-Lambda}
 \gamma_0\Lambda^{2s} > \frac{\omega_{n-1}}{4}\cdot \frac{C_{n,s}}{1-s}\cdot \Big(\frac{2(n-1)}{n}\Big)^{\frac{s}{2}}K,
\end{equation}
we see from \eqref{eq:oplarge} that the matrices in $\A_2\setminus \mathcal{N}(\Lambda)$ do not contribute to the infimum, and thus, 
$$
\K^s u(x) = \frac{C_{n,s}}{2} \inf_{A\in \A_2\cap\mathcal{N}(\Lambda)} \int_{\Rn} \frac{\delta u(x,Ay)}{|y|^{n+2s}}\, dy,
$$
 i.e., $\K^s$ is uniformly elliptic in $\Omega$. Moreover, 
 $\Lambda=O(1)$ as $s\to 1$, since we can choose $\Lambda$ uniformly in \eqref{choose-Lambda}.

It remains to see that the infimum over the compact set $\A_2\cap\mathcal N(\Lambda)$ (closed by Lemma~\ref{lem:char} and bounded) is attained. It is enough to show that
$A\mapsto\mathcal L^s_Au(x)$ is continuous on $\mathcal N(\Lambda)$, which follows from the dominated convergence theorem.
Indeed, by \eqref{alternative.writing.L.A1}--\eqref{alternative.writing.L.A2},  we can write
$$
\mathcal L^s_Au(x)=\frac12\int_{\Rn}\delta u(x,y)\,k_A(y)\,dy,
$$
and by \eqref{alternative.writing.L.A3}, for all $A\in\mathcal N(\Lambda)$ it holds
$$
|\delta u(x,y)k_A(y)|\leq\Lambda^{2n+2s}\,C_{n,s}\,\frac{|\delta u(x,y)|}{|y|^{n+2s}}.
$$  
Note that $\delta u(x,y)\,|y|^{-n-2s}\in L^1(\Rn)$. Its positive part is integrable since $u$ is touched from above by a $C^2$-function at $x$ and
$u\in L^1_{\omega_s}(\Rn)$, and its negative part since $t_*^{\,s}\,\Delta^su(x)=\mathcal L^s_{A_*}u(x)\geq\eta_0$.
Since $k_A(y)$ is continuous in $A$ for each $y\neq0$, dominated convergence gives the continuity, and the infimum is a minimum.
\end{proof}

\section{On the notion of nonlocal semiconcavity}\label{sec:NSC.discussion}

In this section, we discuss further the geometric meaning of condition \eqref{NSC}, and show that the quantity $\mathcal Q^s u(x;y,z)$  contains transversal (i.e., $y,z$ not collinear) information the one-dimensional fractional Laplacian $\Theta^s u(x,e)$ does not.
First, we point out that $\mathcal Q^s u(x;\cdot,\cdot)$ involves second-order differences of the even function $\I^su(x,\cdot)$, which is itself built from second-order differences of $u$.
Altogether, $\mathcal Q^s$ is built from fourth-order differences of $u$, which motivates the following definition:
\begin{defn}
Given $x\in\R^n$ and $h,h'\in\R^n$,
we define the mixed fourth-order difference at $x$ in the directions of $h,h'$ as the second-order difference of $\delta u(\cdot,h')$ at $x$ in the direction of $h$, i.e.,
\begin{equation}\label{def.delta2}
\begin{split}
\delta^2u(x;h,h')&:=\delta\big(\delta u(\cdot,h')\big)(x,h)\\
&=\delta u(x+h,h')+\delta u(x-h,h')-2\delta u(x,h')\\
&=u(x+h+h')+u(x+h-h')+u(x-h+h')+u(x-h-h')\\
&\hspace{75pt}
-2u(x+h)-2u(x-h)-2u(x+h')-2u(x-h')+4u(x).
\end{split}
\end{equation}
We observe that the fourth-order difference $\delta^2u(x;h,h')$ is symmetric in $(h,h')$. 
\end{defn}

\begin{rem}
For $u$ smooth enough, the function $\delta^2u(x;h,h')$ can be seen as the standard finite-difference discretization of the mixed fourth
derivative $\partial_h^2\partial_{h'}^2u(x)=D^4u(x)[h,h,h',h']$, sampled on the nine-point stencil $\{x\pm h\pm h',\ x\pm h,\ x\pm h',\ x\}$, whose four outer points $x\pm h\pm h'$ are the vertices of a parallelogram. It is worth mentioning that  the fourth-order differences  $\delta^2u$ also appear indirectly in the proof of the nonlocal Evans--Krylov theorem \cite{Caffarelli.Silvestre2}. A key step in their argument is to show that second differences $\delta u(\cdot,h')$ of solutions of concave
equations are subsolutions of the extremal operators, which ultimately involves integrals of $\delta^2u$.
\end{rem}

The mixed fourth-order difference $\delta^2u$ has an interpretation as the defect in the parallelogram identity for the second differences of $u$. Indeed, recall that a continuous function satisfies the Jordan–von Neumann functional equation, or parallelogram identity,
\begin{equation}\label{paralellogram.f} f(y+z)+f(y-z)=2f(y)+2f(z),\qquad\textrm{for all}\  y,z\in\R^n
\end{equation}
if and only if it is a quadratic form, see \cite{JvN,Aczel-Dhombres}.
Then, since for every function $u:\Rn\to\R$  it holds
\begin{equation}\label{eq:parallelogram.delta.u}
\delta u(x,y+z)+\delta u(x,y-z)
=2\,\delta u(x,y)+2\,\delta u(x,z)+\delta^2u(x;y,z),
\qquad
\textrm{for all}\ x,y,z\in\Rn,
\end{equation}
the fourth-order difference $\delta^2u(x;y,z)$ measures how far the even
function $\delta u(x,\cdot)$ is from being a quadratic form.

The following result quantifies the defect in the parallelogram identity for $\I^su$, with a remainder built from fourth-order differences of $u$ that vanishes as $s\to1$. Note that the pointwise statement below concerns $\I^su(x,\cdot)$ along a fixed direction, and therefore requires some control of $u$ at infinity that cannot be provided by the averaged condition $u\in L^1_{\omega_s}(\Rn)$. We assume, for simplicity, a polynomial growth at infinity.

\begin{lem}\label{lem:second.diff.Is}
Let $u\in C(\Rn)\cap L^1_{\omega_s}(\Rn)$, $x\in\Rn$, and $y,z\in\Rn$.  Then, whenever $\I^s u(x,y\pm z)$, $\I^s u(x,y),$  and $\I^s u(x,z)$ are all finite, 
\begin{equation}\label{eq:second.diff.Is}
\I^s u(x,y+z)+\I^s u(x,y-z)=2\,\I^s u(x,y)+2\,\I^s u(x,z)+\mathcal R^s u(x;y,z),
\end{equation}
where
\begin{equation}\label{def.Rs}
\mathcal R^s u(x;y,z):=2(1-s)\int_0^\infty\frac{\delta^2u(x;ry,rz)}{r^{1+2s}}\,dr .
\end{equation}
Moreover, assume that $u\in C^2$ in a neighborhood of $x$, and that there exist $C>0$
and $\beta\in[0,2)$ with $|u(\xi)|\leq C(1+|\xi|^{\beta})$ for all $\xi\in\Rn$. Then
\[
\lim_{s\to1}\mathcal R^s u(x;y,z)=0\qquad\textrm{for all }y,z\in\Rn .
\]
\end{lem}

\begin{proof}
By \eqref{expansion.second.differences} and \eqref{def.delta2},
\[
\begin{split}
\delta u(x,r(y+z))+\delta u(x,r(y-z))
&=2\,\delta u(x,ry)+ \delta u(x+ry,rz)+\delta u(x-ry,rz)\\
&= 2\delta u(x,ry) +2\,\delta u(x,rz) +\delta^2u(x;ry,rz).
\end{split}
\]
Dividing both sides by $r^{1+2s}$ and integrating, we get \eqref{eq:second.diff.Is}.

Let us now prove that $\mathcal R^su(x;y,z)\to0$ as $s\to1$. We may assume
$y,z\neq0$, since otherwise $\delta^2u(x;ry,rz)\equiv0$. Let $\rho>0$ be such that
$u\in C^2(B_{2\rho}(x))$, and let $\omega$ denote a modulus of continuity of $D^2u$
on $B_{2\rho}(x)$. Fix $r_0\in(0,1]$ with $r_0(|y|+|z|)\leq\rho$, and split
\[
\mathcal R^su(x;y,z)=2(1-s)\int_0^{r_0}\frac{\delta^2u(x;ry,rz)}{r^{1+2s}}\,dr
+2(1-s)\int_{r_0}^{\infty}\frac{\delta^2u(x;ry,rz)}{r^{1+2s}}\,dr
=: I_1+I_2 .
\]
For $I_1$, observe that when $\xi\in B_\rho(x)$ and $|h'|\leq\rho$, by Taylor's formula
\[
\delta u(\xi,h')=\int_0^1(1-t)\big(\langle D^2u(\xi+th')h',h'\rangle+\langle D^2u(\xi-th')h',h'\rangle\big)\,dt,
\]
so that, for $\xi_1,\xi_2\in B_\rho(x)$,
$|\delta u(\xi_1,h')-\delta u(\xi_2,h')|\leq|h'|^2\,\omega(|\xi_1-\xi_2|)$. Applying this with $h'=rz$ and $\xi_1=x\pm ry$, $\xi_2=x$, we obtain
\[
|\delta^2u(x;ry,rz)|\leq 2\,r^2|z|^2\,\omega(r|y|),\qquad 0<r\leq r_0 .
\]
Hence, for any $\varepsilon\in(0,r_0)$,
\[
|I_1|\leq 4(1-s)|z|^2\int_0^{r_0}\frac{\omega(r|y|)}{r^{2s-1}}\,dr
\leq 4|z|^2\bigg(\frac{\omega(\varepsilon|y|)\,\varepsilon^{2-2s}}{2}
+\omega(r_0|y|)\,(1-s)\int_\varepsilon^{r_0}r^{1-2s}\,dr\bigg).
\]
As $s\to1$, the first term is bounded, while the second one, for fixed $\varepsilon$,
\[
(1-s)\int_\varepsilon^{r_0}r^{1-2s}\,dr=\frac12\big(r_0^{2-2s}-\varepsilon^{2-2s}\big)\to0
\qquad
\textrm{as}\ s\to1.
\]
Therefore, $\limsup_{s\to1}|I_1|\leq2|z|^2\omega(\varepsilon|y|)$ for every
$\varepsilon>0$, and letting $\varepsilon\to0$ gives $I_1\to0$.

For $I_2$, the growth assumption yields, for $r\geq r_0$,
\[
|\delta^2u(x;ry,rz)|\leq 16\,C\Big(1+(|x|+r|y|+r|z|)^{\beta}\Big)
\leq C'\,r^{\beta},
\]
with $C'$ independent of $s$. Thus, for $s>\beta/2$,
\[
|I_2|\leq 2(1-s)\,C'\int_{r_0}^\infty r^{\beta-1-2s}\,dr
=\frac{2(1-s)\,C'\,r_0^{\beta-2s}}{2s-\beta}\longrightarrow0
\qquad\textrm{as }s\to1,
\]
since $\beta<2$. This completes the proof.
\end{proof}

The defect $\mathcal R^s u$ has no sign in general, not even for convex $u$, since convexity imposes no sign on fourth derivatives.
Let us also point out that classical convexity is a strong assumption in this
context, since nonconstant convex functions grow at least linearly, which requires $s>1/2$.

\begin{rem}\label{rem:NSC.vs.semiconcavity}
 By \eqref{def.alternate.I.s} and \eqref{eq:second.diff.Is}, condition
\eqref{NSC} is equivalent to
\[
|z|^{2s}\,\Theta^s u\Big(x,\tfrac{z}{|z|}\Big)+\tfrac12\,\mathcal R^su(x;y,z)\leq K|z|^{2s},
\qquad y,z\in\Rn,
\]
and, under the assumptions of the lemma, letting $s\to1$ and using
$\Theta^su(x,e)\to u_{ee}(x)$ we recover the classical condition
$\langle D^2u(x)z,z\rangle\leq K|z|^2$. 
\end{rem}

In view of the notion of fractional convexity introduced in \cite{DelPezzo-Quaas-Rossi}, see \eqref{FCX},  
a reasonable candidate for a definition of fractional semiconcavity would be the
boundedness from above of the one-dimensional fractional Laplacians, i.e.,
\begin{equation}\label{FSC}
    \Theta^su(x,e)\leq K,\qquad\textrm{for all}\ e\in\mathbb S^{n-1}.
\end{equation}
In this line, Remark~\ref{rem:NSC.vs.semiconcavity}, could suggest that Theorem~\ref{thm:main} might hold replacing \eqref{NSC} by \eqref{FSC}. However, this is not the case, and we will provide a counterexample in Section \ref{sec:counterexample}.

Another way to see this fact is given in the lemma below, where we show that the term $\mathcal R^s$ plays a crucial role in the above Proposition~\ref{prop:tech-lem}. More precisely, in the local argument of Section~\ref{sec:idealocal}, the key step
\eqref{comentario.sobre.semiconcavidad.Laplacian} rests on the fact that the
map
\[
Q\longmapsto \mathcal L_Q u(x)=\tr\big(Q^2D^2u(x)\big)
\]
is linear in $Q^2$. This allows one to split a coefficient matrix into
the component along the asymptotic cone and the axial component and treat them separately. In the nonlocal
setting, the corresponding map $Q\mapsto\mathcal L^s_Qu(x)$ is not linear in
$Q^2$, but satisfies the operator comparison estimate in Proposition~\ref{prop:tech-lem}. In the next lemma, we will quantify the estimate through the defect $\mathcal R^s$.

\begin{lem}\label{lem:averaged.parallelogram}
Let $s\in(0,1)$, and let $u\in C(\Rn)\cap L^1_{\omega_s}(\Rn)$ satisfy \eqref{NSC} at a point $x$ where a $C^2$-function touches $u$ from above. Let
$Q_1,Q_2\in\S^n_+$ with $Q_2^2>Q_1^2$, and assume
\begin{equation}\label{eq:finiteness.three}
\mathcal L^s_{Q_1}u(x),\ \mathcal L^s_{Q_2}u(x),\ \mathcal L^s_{(Q_2^2-Q_1^2)^{1/2}}u(x)>-\infty .
\end{equation}
Then, $\mathcal R^su(x;y,z)$ is finite for a.e.~$(y,z)$, and integrable against $G_{Q_1^2}(y)G_{Q_2^2-Q_1^2}(z)$. Moreover, we have the identity
\begin{equation}\label{eq:additivity.defect}
\mathcal L^s_{Q_2}u(x)-\mathcal L^s_{Q_1}u(x)
=\mathcal L^s_{(Q_2^2-Q_1^2)^{1/2}}u(x)+\overline{\mathcal R^s}_{Q_1,Q_2}u(x)
\end{equation}
for
\begin{equation}\label{def.averaged.Rs}
\overline{\mathcal R^s}_{Q_1,Q_2}u(x)
:=\frac{C_{n,s}}{8(1-s)\kappa_{n,s}}\int_{\Rn}\int_{\Rn}\mathcal R^su(x;y,z)\,G_{Q_1^2}(y)\,G_{Q_2^2-Q_1^2}(z)\,dy\,dz.
\end{equation}
\end{lem}

\begin{proof}
As in the proof of Proposition~\ref{prop:tech-lem}, the hypotheses imply that $\I^su(x,\cdot)\in L^1(\mathbb S^{n-1})$ and finite a.e. by homogeneity. In particular, $\I^su(x,\cdot)$ is integrable against $G_{Q_2^2}$, $G_{Q_1^2}$ and $G_{Q_2^2-Q_1^2}$.
Since $G_{Q_2^2-Q_1^2}$ is even, Lemma~\ref{lem:gauss-conv}, Corollary~\ref{lem:gauss-sphere}, and the representation formula  of Lemma~\ref{lem:newrep}, yield
\[
\begin{split}
\int_{\Rn}\int_{\Rn}\I^su(x,y\pm z)\,&G_{Q_1^2}(y)G_{Q_2^2-Q_1^2}(z)\,dy\,dz
=\int_{\Rn}\I^su(x,w)\,G_{Q_2^2}(w)\,dw\\
&=\kappa_{n,s}\int_{\mathbb S^{n-1}}\I^su(x,Q_2\theta)\,d\H^{n-1}_\theta
=
\frac{4(1-s)\,\kappa_{n,s}}{C_{n,s}}
\,\mathcal L^s_{Q_2} u(x).
\end{split}
\]
Similarly, using that each Gaussian has unit mass, Corollary~\ref{lem:gauss-sphere}, and the representation formula, we obtain
\[
\int_{\Rn}\int_{\Rn}\I^su(x,y)\,G_{Q_1^2}(y)G_{Q_2^2-Q_1^2}(z)\,dy\,dz
=
\int_{\Rn}\I^su(x,y)\,G_{Q_1^2}(y)\,dy
=\frac{4(1-s)\,\kappa_{n,s}}{C_{n,s}}\mathcal L^s_{Q_1} u(x),
\qquad
\]
and
\[
\int_{\Rn}\int_{\Rn}\I^su(x,z)\,G_{Q_1^2}(y)G_{Q_2^2-Q_1^2}(z)\,dy\,dz
=\frac{4(1-s)\,\kappa_{n,s}}{C_{n,s}}\mathcal L^s_{(Q_2^2-Q_1^2)^{1/2}} u(x).
\]
Hence, all four terms in \eqref{eq:second.diff.Is} are integrable against $G_{Q_1^2}(y)G_{Q_2^2-Q_1^2}(z)$, so $\mathcal R^su(x;y,z)$ is too, and integrating
\eqref{eq:second.diff.Is} gives
\eqref{eq:additivity.defect}.
\end{proof}

We can recover the comparison estimate \eqref{comparison.estimate} from \eqref{eq:additivity.defect} as follows.
By \eqref{eq:second.diff.Is}, condition \eqref{NSC} reads $2\,\I^su(x,z)+\mathcal R^su(x;y,z)\leq2K|z|^{2s}$. Averaging this inequality against $G_{Q_1^2}(y)\,G_{Q_2^2-Q_1^2}(z)$, and using \eqref{def.averaged.Rs}, by the computations in the proof of Lemma~\ref{lem:averaged.parallelogram}, and \eqref{lem:comp.eq3} with $Q=(Q_2^2-Q_1^2)^{1/2}$, 
\[
\mathcal L^s_{(Q_2^2-Q_1^2)^{1/2}}u(x)+\overline{\mathcal R^s}_{Q_1,Q_2}u(x)
\leq K\cdot\frac{\omega_{n-1}}{4 n^s}\cdot\frac{C_{n,s}}{1-s}\cdot\big(\tr(Q_2^2-Q_1^2)\big)^s
=c_{K,s}\big(\tr(Q_2^2-Q_1^2)\big)^s.
\]

\medskip

\section{The counterexample} \label{sec:counterexample}

In the following, we show that condition \eqref{weaker-theta} is too weak to obtain the conclusion of Theorem~\ref{thm:main}, which motivates our main hypothesis \eqref{NSC}.
The idea is to construct a function whose nonlocal Hessian form is bounded on the sphere, but vanishes on $\{y_1=0\}$. Then, one can find a family of degenerate matrices in $\mathcal A_2$ that contract the $e_1$-direction and expand the transverse directions, while keeping $\mathcal L_A^s u(0)$ bounded. This shows that directional control alone is not enough, and that the transversal information in \eqref{NSC} is essential.

\begin{prop}\label{prop:counterexample}
There exists $u\in C_c(\mathbb{R}^n)$, given precisely in \eqref{eq:realization} below, such that $u(x)$ can be touched from above at $x=0$ by a $C^2$-function, and it satisfies
\begin{enumerate}[(i)]\itemsep3pt
\item $\K^su(0)=\eta_0$, for some $\eta_0>0$;
\item $0\le \Theta^s u(0,e)\le 1,$  for all  $e\in\mathbb{S}^{n-1}$.
\end{enumerate}
However, for every $\Lambda\ge 1$,
\[
\inf_{A\in\A_2\cap\mathcal N(\Lambda)}\mathcal L^s_Au(0)>\inf_{A\in\A_2}\mathcal L^s_Au(0)=\K^su(0).
\]
Moreover, the key estimate of Theorem~\ref{thm:keyest} fails at $x=0$, i.e., there exists a degenerating matrix family $\{A_\lambda\}\subset \A_2$ such that
$\mathcal L^s_{A_\lambda}u(0)$ remains bounded as $\|A_\lambda\|_F\to\infty$.
\end{prop}

\begin{rem}
Observe that for the function $u$ given in the Proposition \ref{prop:counterexample}, we have:
\begin{enumerate}
\item[{(a)}] 
By property \emph{(ii)}, $u$ is fractionally convex at $x=0$ in the sense of \cite{DelPezzo-Quaas-Rossi}, and satisfies the one-sided bound \eqref{weaker-theta} with $K=1$.

\item[(b)]  Hypothesis \eqref{NSC} fails at $x=0$ through the cross-direction term, i.e., for $y\perp e_1$, $z=te_1$, and $t>0$,
\begin{equation*}
\mathcal Q^su(0;y,z) =2\,t^s\,(|y|^2+t^2)^{s/2}\geq 2\,t^s|y|^s,
\end{equation*}
which is not bounded by $2K t^{2s}$ for any $K$, as $|y|\to\infty$.
\end{enumerate}
\end{rem}

The argument in Proposition \ref{prop:counterexample} is based on a probabilistic interpretation of the representation formula \eqref{representation-L}, i.e.,
\begin{equation}\label{formula:contraejemplo}
\mathcal L^s_Au(0)=
\frac{C_{n,s}}{4(1-s)\kappa_{n,s}}
\int_{\Rn}\I^su(0,Ay)G(y)\,dy=:\frac{C_{n,s}}{4(1-s)\kappa_{n,s}}F(A),
\end{equation}
where $G(y)$ is the Gaussian density in $\Rn$, given by
$$
G(y):=
(2\pi)^{-\frac{n}{2}}\, e^{-\frac{|y|^2}{2}}.
$$
More precisely, we will construct a function $u$ such that, at the origin, $\I^s u(0,y)$ is a product of two factors. Then, the  Gaussian correlation inequality allows us to separate them into the product of two integrals with an inequality in the right direction. In addition, Lemma \ref{lema1} provides uniform control from below for every $A\in \mathcal A_2$.

 In particular, writing the coordinates as $y=(y_1,y')\in\R\times\R^{n-1}$, we will choose 
\begin{equation*}
\I^s u(0,y)=|y_1|^s|y|^s,
\end{equation*}
and show that \begin{equation*}
F(A)>F_\infty:=\int_{\Rn}|y_1|^s|y'|^s G(y)\,dy,
\end{equation*}
for every $A\in\mathcal A_2$. This is done in Lemma \ref{lem:bound-below}.

The crucial observation here is that decoupling the first coordinate from the rest allows us to isolate the first direction in the integrals. Then, we can construct a suitable sequence of matrices $A_\lambda\in\mathcal A_2$  with eigenvalues $(\lambda_1,\lambda,\ldots,\lambda)$ with 
$\lambda_1\to 0$, $\lambda\to\infty$, that degenerates exactly as described in the first case of Corollary~\ref{cor:degeneracies}. Since quantities in the direction of $y_1$ degenerate to zero, we will be able to show that $F(A_\lambda)\to F_\infty$. Taking into account that $\|A_\lambda\|_F\to\infty$, it yields a  counterexample for the key estimate of Theorem~\ref{thm:keyest}. This is done in Lemma~\ref{lemma:counterexample-key-estimate} below.

\begin{lem}\label{lemma:u-Phi}
There exists $u\in C_c(\Rn)$ such that $u(x)$ can be touched from above at $x=0$ by a $C^2$-function, and
\begin{equation}\label{relation-Phi}
\I^s u(0,y)=|y_1|^s|y|^s.
\end{equation}
\end{lem}

\begin{proof}
 Fix $\rho\in(0,1]$ with $\rho^2<\frac{1}{2(1-s)}$. Let $\chi\in C_c([0,\infty))$ be nonnegative, with $\chi(t)=t^2$ for
$0\leq t\leq\rho$, normalized so that
\begin{equation}\label{eq:chi.normalization}
4(1-s)\int_0^\infty\frac{\chi(t)}{t}\,dt=1 ,
\end{equation}
and define
\begin{equation}\label{eq:realization}
u(y):=|y_1|^s|y|^s\chi(|y|),\quad \text{ for } y\in\Rn.
\end{equation}
By construction, $u(0)=0$ and
$$
u(y)\leq |y|^{2s+2}\leq  |y|^2, \quad \text{ for all } |y|\leq \rho.
$$
Hence, the function $|y|^2$ touches $u(y)$ from above at $y=0$.

Since $u$ is even, $\delta u(0,ry)=2u(ry)$, and therefore,
\[
\I^su(0,y)=2(1-s)\int_0^\infty\frac{\delta u(0,ry)}{r^{1+2s}}\,dr
=4(1-s)|y_1|^s|y|^s\int_0^\infty\frac{\chi(r|y|)}{r}\,dr=|y_1|^s|y|^s,
\]
and this concludes the proof of the lemma.
\end{proof}

As mentioned above, the proof of Proposition \ref{prop:counterexample} relies on the  representation formula from \eqref{representation-L},  
\begin{equation}\label{formula-F}
\begin{split}
\mathcal L^s_Au(0)&=
\frac{C_{n,s}}{4(1-s)\kappa_{n,s}} \int_{\Rn}\I^su(0,Ay)G(y)\,dy\\
&=\frac{C_{n,s}}{4(1-s)\kappa_{n,s}} \int_{\Rn}|(Ay)_1|^s|Ay|^s G(y)\,dy=:\frac{C_{n,s}}{4(1-s)\kappa_{n,s}}F(A)
\end{split}
\end{equation}
for all $A\in\S^n_+$.

We denote by  $G^{(k)}$ the standard Gaussian density on $\R^k$ and define
\begin{equation}\label{def.Linfty}
F_\infty:=\int_{\Rn}|y_1|^s|y'|^s G(y)\,dy=\int_{\R}|t|^sG^{(1)}(t)\,dt\cdot\int_{\R^{n-1}}|y'|^sG^{(n-1)}(y')\,dy'.
\end{equation}

\begin{lem}\label{lem:bound-below}
For every $A\in\A_2$, it holds that $F(A)>F_\infty$. 
\end{lem}

\begin{proof}
Let $A\in\A_2$ with eigenvalues $\lambda_1\leq\lambda_2\leq\cdots\leq\lambda_n$ and orthonormal eigenvectors $v_1,\ldots,v_n$. Write $a:=Ae_1$, so that
$|(Ay)_1|^s |Ay|^s=|\langle y,a\rangle|^s\,|Ay|^s$. The first factor is even with sublevel sets that are symmetric slabs orthogonal to $a$, and the second is
even with ellipsoidal sublevel sets. By Lemma~\ref{lemma:KS-version} in Appendix~\ref{appendix}, and the rotation invariance of $G$,
\begin{equation}\label{eq:GCI.applied}
F(A)\geq\int_{\mathbb R^n}|\langle y,a\rangle|^sG(y)\,dy\cdot\int_{\mathbb R^n}|Ay|^s G(y)\,dy
=|a|^s\int_{\mathbb R^n}|y_1|^s G(y)\,dy\cdot\int_{\mathbb R^n}|Ay|^s G(y)\,dy .
\end{equation}
Since $|a|=|Ae_1|\geq\lambda_1$, it remains to bound the last integral from below. Let $P$ be the orthogonal projection onto $v_1^\perp$, which commutes with $A$. Then, for every $y$ with $\langle y,v_1\rangle\neq0$,
\[
|Ay|^2=\lambda_1^2\langle y,v_1\rangle^2+|APy|^2
>|APy|^2\geq\lambda_2^2\,|Py|^2 ,
\]
and the exceptional hyperplane has measure zero. Rotating $v_1$ to $e_1$,
$\int|Py|^s G(y)\,dy=\int|y'|^s G(y)\,dy$, so that
$$\int_{\mathbb R^n}|Ay|^s G(y)\,dy>\lambda_2^s\int_{\mathbb R^n}|y'|^s G(y)\,dy.$$ Therefore, by
\eqref{eq:GCI.applied}, 
\[
F(A)>\lambda_1^s\lambda_2^s\int_{\mathbb R^n}|y_1|^s G(y)\,dy\int_{\mathbb R^n}|y'|^s G(y)\,dy.
\]
Then, decomposing $G(y)=G^{(1)}(y_1)G^{(n-1)}(y')$, by Fubini, we reach
\begin{equation*}
F(A)>\lambda_1^s\lambda_2^sF_\infty\geq F_\infty,
\end{equation*}
where in the last inequality we have used Lemma~\ref{lema1}.
\end{proof}

\begin{lem}\label{lemma:counterexample-key-estimate}
Given $\lambda>0$, let
$A_\lambda$ be the diagonal matrix with eigenvalues $(\lambda_1,\lambda,\ldots,\lambda)$, where $\lambda_1$ is chosen in terms of $\lambda$ so that $A_\lambda\in\A_2$. Then
$F(A_\lambda)\to F_\infty$ and $\|A_\lambda\|_F\to\infty$ as $\lambda\to\infty$.
\end{lem}

\begin{proof} 
The eigenvalues $\mu=(\mu_1,\mu,\ldots,\mu)$ of
$A_\lambda^2$, with $\mu_i=\lambda_i^2$, satisfy \eqref{eq:classA2}, which reduces
to
\[
2\mu_1\mu-\frac{n-2}{n-1}\,\mu_1^2=2 .
\]
Hence, as $\lambda\to\infty$, we have $\mu_1\to0$ and $\mu_1\mu\to1$ as $\mu\to\infty$, i.e.,
$\lambda_1\lambda\to1$ and $\lambda_1/\lambda\to0$, which is the first degeneracy described in Corollary \ref{cor:degeneracies}. Then
\[
F(A_\lambda)=\int_{\mathbb R^n}|\lambda_1y_1|^s\big(\lambda_1^2y_1^2+\lambda^2|y'|^2\big)^{s/2}G(y)\,dy
=(\lambda_1\lambda)^s\int_{\mathbb R^n}|y_1|^s\Big(|y'|^2+\frac{\lambda_1^2}{\lambda^2}\,y_1^2\Big)^{s/2}G(y)\,dy
\longrightarrow F_\infty
\]
by dominated convergence, since the integrand is bounded by $|y_1|^s|y|^sG(y)$
for $\lambda\geq\lambda_1$. Clearly, $\|A_\lambda\|_F^2\geq\lambda^2\to\infty$.
\end{proof}

Finally, we give the proof of the counterexample.

\begin{proof}[Proof of Proposition~8.1]
Let $u$ be the function constructed in Lemma~\ref{lemma:u-Phi}. 
By \eqref{formula-F}, for $A\in \S^n_+,$
\[
    \mathcal L_A^s u(0) = \frac{C_{n,s}}{4(1-s)\kappa_{n,s}} F(A).
\]
By Lemmas~\ref{lem:bound-below} and \ref{lemma:counterexample-key-estimate}, we have
\[
    F(A)>F_\infty,
    \quad\text{for all } A\in \A_2,
\]
and there is a sequence $\{A_\lambda\}\subset\A_2$ such that
\[
    F(A_\lambda)\longrightarrow F_\infty, \quad \text{ as } \lambda \to \infty.
\]
Hence,
\[
    \mathcal K^s u(0)= \inf_{A\in\mathcal{A}_2} \mathcal L_A^s u(0)
    =\frac{C_{n,s}}{4(1-s)\kappa_{n,s}}\,F_\infty =:\eta_0>0.
\]

Now, fix $\Lambda\ge 1$. Since $\mathcal{A}_2\cap\mathcal{N}(\Lambda)$ is compact and $F$ is continuous, Lemma~\ref{lem:bound-below} implies
\[
    \min_{A\in\mathcal{A}_2\cap\mathcal{N}(\Lambda)}F(A)>F_\infty.
\]
It follows that
\[
    \inf_{A\in\mathcal{A}_2\cap\mathcal{N}(\Lambda)} \mathcal L_A^s u(0)
    > \inf_{A\in\mathcal{A}_2} \mathcal L_A^s u(0) = \mathcal K^s u(0),
\]
which proves that the conclusion of Theorem~1.2 fails at the origin for
every $\Lambda\ge1$.

Finally, Lemma~8.5 also gives
\[
    \mathcal L_{A_\lambda}^s u(0) \longrightarrow \frac{C_{n,s}}{4(1-s)\kappa_{n,s}}\,F_\infty,
    \quad \text{ as } \|A_\lambda\|_F\to \infty.
\]
Therefore, the key estimate of Theorem~6.1 does not hold for the function $u$.
\end{proof}

\appendix

\section{Some results in probability}\label{appendix}

In the proof of Lemma~\ref{lem:bound-below}, we used an inequality due to Khatri \cite{Khatri} and \v{S}id\'ak \cite{Sidak} (who obtained this result independently), which is a particular case of the Gaussian correlation inequality for symmetric convex sets. The general inequality was conjectured in \cite{correlation-conjecture} and proved in \cite{Royen} (see the modern presentation in \cite{Latala-Matlak}).

Given a unit vector $\nu\in\mathbb S^{n-1}$ and $t\in[0,\infty]$, the \emph{symmetric slab} of width $t$ orthogonal to $\nu$ is
\[
S_\nu(t):=\{y\in\Rn:\ |\langle y,\nu\rangle|\leq t\} ,
\]
with the convention $S_\nu(\infty)=\Rn$. Slabs are convex and symmetric with respect to the origin.

\begin{lem} If
$K\subset\Rn$ is closed, convex, and symmetric with respect to the origin, and $S=\{y\in\Rn:|\langle y,\nu\rangle|\leq t\}$ is a symmetric slab, then
\begin{equation}\label{eq:sidak}
\int_{K\cap S}G\,dy\ \geq\ \int_KG\,dy\int_SG\,dy .
\end{equation}
\end{lem}

In probabilistic language, this tells us that
$$\mathbb P(K\cap S)\geq  \mathbb P(K) \mathbb P(S),$$ where the probability is calculated with respect to the Gaussian measure. Equivalently,
$$\mathbb P(K|S)\geq \mathbb P(K),$$
whenever $\mathbb P(S)>0$, that is, knowing that you are in the slab $S$ only increases the probability of being in the convex set $K$, regardless of the relative orientation of the two sets, see Figure~\ref{fig:sidak}.

\begin{figure}[t]
\centering
\begin{tikzpicture}[scale=1]

\begin{scope}
  \clip (-4.6,-2.9) rectangle (4.6,2.9);

  \fill[blue,opacity=0.13] (-1.2,-2.9) rectangle (1.2,2.9);

  \begin{scope}
    \fill[red,opacity=0.14] (0,0) ellipse [x radius=2.6, y radius=1.45, rotate=22];
  \end{scope}

  \begin{scope}
    \clip (-1.2,-2.9) rectangle (1.2,2.9);
    \fill[purple,opacity=0.34] (0,0) ellipse [x radius=2.6, y radius=1.45, rotate=22];
  \end{scope}

  \foreach \r in {0.85,1.6,2.35}{
    \draw[gray!75,dashed,line width=0.35pt] (0,0) circle (\r);
  }

  \draw[black,line width=0.7pt] (-1.2,-2.9) -- (-1.2,2.9);
  \draw[black,line width=0.7pt] (1.2,-2.9) -- (1.2,2.9);
  \draw[black,line width=0.7pt] (0,0) ellipse [x radius=2.6, y radius=1.45, rotate=22];
  \draw[black,line width=0.3pt] (-3,0) -- (3,0);
  \draw[black,line width=0.3pt] (0,-4) -- (0,4);
\end{scope}

\fill (0,0) circle (1.4pt);

\node[font=\footnotesize] at (-0.9,2.5){$S$};

\node[font=\footnotesize] at (1.6,1.2){$K$};

\node[anchor=west,align=left,font=\footnotesize] at (0,0.5){$K \cap S$};

\draw[black,line width=0.3pt,->] (-2.5,0.9) -- (-1.5,0.6);
\node[anchor=east,align=right,font=\footnotesize] at (-2.5,1){Gaussian level sets};


\end{tikzpicture}

\caption{The Khatri--\v{S}id\'ak inequality \eqref{eq:sidak}.}
\label{fig:sidak}
\end{figure}
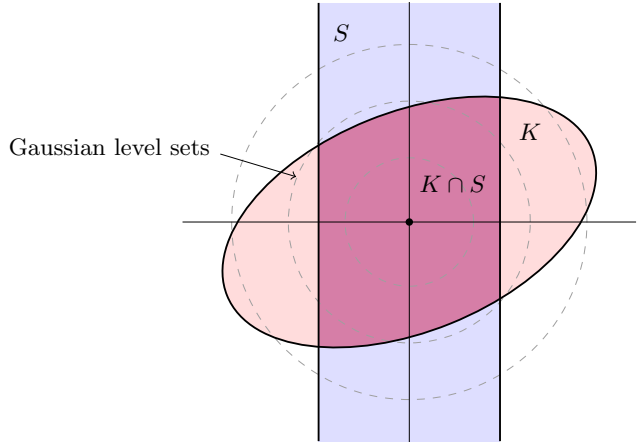

The following is the functional form of \eqref{eq:sidak}. Although it is not the most general statement, it suffices for our purposes.

\begin{lem}\label{lemma:KS-version}
Let $f_1,f_2:\Rn\to[0,\infty)$. Assume that $f_1(y)=\phi(|\langle y,\nu\rangle|)$ for some nondecreasing
continuous $\phi:[0,\infty)\to[0,\infty)$ and $\nu\in\mathbb S^{n-1}$, and that  $f_2$ is even and continuous, and the sublevel set $\{f_2\leq t\}$ is convex for every $t\geq0$. Then
\begin{equation}\label{eq:KS.functional}
\int_{\Rn}f_1f_2G\,dy\ \geq\ \int_{\Rn}f_1G\,dy\int_{\Rn}f_2G\,dy,
\end{equation}
both sides being possibly infinite.
\end{lem}

\begin{proof}
Fix $M>0$ and set $g_i:=\min(f_i,M)$, $h_i:=M-g_i\in[0,M]$. For $t\in(0,M)$, we have $\{h_i>t\}=\{f_i<M-t\}$, while $\{h_i>t\}=\emptyset$ for $t\geq M$. Since $\phi$ is nondecreasing, $\{h_1>t\}$ is a symmetric slab orthogonal to $\nu$. Moreover, $\{h_2>t\}$ is convex, as it is an increasing union of the convex sets $\{f_2\leq M-t-\frac1k\}$, and symmetric because $f_2$ is even. Replacing these sets by their closures does not change the relevant Gaussian measures, since the boundary of a convex set has zero Lebesgue measure. Therefore, \eqref{eq:sidak} applies to the closures.

By the layer-cake formula and Tonelli, since  all integrands are nonnegative,
\begin{equation}\label{layer-cake}
\begin{split}
\int_{\Rn} h_1h_2\,G\,dy
&=\int_0^M\int_0^M\Big(\int_{\{h_1>t\}\cap\{h_2>\tau\}}G\Big)dt\,d\tau\\
&\geq\int_0^M\int_0^M\Big(\int_{\{h_1>t\}}G\Big)\Big(\int_{\{h_2>\tau\}}G\Big)dt\,d\tau\\
&=\int_{\Rn} h_1G\int_{\Rn} h_2G .
\end{split}
\end{equation}
Substituting $h_i=M-g_i$ and using $\int_{\Rn}G\,dy=1$, all terms involving $M$ are finite (as $0\leq g_i\leq M$) and cancel, leaving
\[
\int_{\Rn} g_1g_2\,G\,dy\ \geq\ \int_{\Rn} g_1G\,dy\int_{\Rn} g_2G\,dy .
\]
Finally $g_i\uparrow f_i$ pointwise as $M\uparrow\infty$, and since $g_1,g_2\geq0$ are both nondecreasing in $M$, also $g_1g_2\uparrow f_1f_2$.
By monotone convergence in each of the three integrals, we obtain \eqref{eq:KS.functional}.
\end{proof}

\begin{rem}
Since $G\,dy$ is a probability measure, claim \eqref{eq:KS.functional} is equivalent to $\mathrm{Cov}_G(f_1,f_2)\geq0$. Thus, the previous lemma states that two nonnegative functions whose sublevel sets are symmetric slabs and symmetric convex sets, respectively, are nonnegatively correlated with respect to the Gaussian measure. The layer-cake formula that we use in  \eqref{layer-cake} is a classical identity by Hoeffding \cite{Hoeffding-aleman} (translated in \cite{Hoeffding}, see also the more modern \cite{Lehmann}), who showed that for $X,Y$  real random variables with $\mathbb E[X],\mathbb E[Y],\mathbb E[XY]<\infty$, then
\begin{equation}\label{eq:hoeffding}
\mathrm{Cov}(X,Y)=\int_{\R}\int_{\R}\Big[F_{X,Y}(t,t')-F_X(t)F_Y(t')\Big]\,dt\,dt',
\end{equation}
and the double integral converges absolutely. Here, we have denoted by $F_{X,Y}$ their joint distribution function and by $F_X,F_Y$ the
marginal ones. 
\end{rem}

\subsection*{Disclaimer on the use of AI tools}
The authors used \texttt{Claude Fable 5} to assist with the probabilistic interpretation and the construction of the counterexample in Section~\ref{sec:counterexample}. All such contributions were independently verified by the authors, who take full responsibility for the contents of the paper.

\bigskip
\subsection*{Acknowledgements}

F.C., M.G., and M.S.C. acknowledge financial support from the grant PID2023-150166NB-I00 (MICIU/AEI/10.13039/501100011033).

M.G. also acknowledges financial support from the ``Severo Ochoa Programme for Centers of Excellence in R\&D'' (CEX2019-000904-S) and RED2024-153842-T.

M.S.C. also acknowledges financial support from the Juan de la Cierva grant JDC2023-050724-I (MICIU/AEI/10.13039/501100011033 and FSE+).

\medskip


\bibliographystyle{plain}

\end{document}